\documentclass[12pt,reqno]{amsart}
\usepackage[latin1]{inputenc}
\usepackage[T1]{fontenc}

\usepackage{graphicx,amsmath,amsthm,amssymb,fullpage}
\usepackage{epsf, subfigure, verbatim}

\usepackage{srcltx}

\usepackage{amsthm}

\theoremstyle{plain}
\newtheorem{X}{X}[section]

\newtheorem{definition}[X]{Definition}

\newtheorem{lemma}[X]{Lemma}
\newtheorem{proposition}[X]{Proposition}
\newtheorem{theorem}[X]{Theorem}

\newtheorem{hypothesis}[X]{Hypothesis}
\newtheorem*{hyp}{Hypothesis \ref{bombieri-vinogradov}*}
\newtheorem*{rem}{Remark}
\newtheorem*{thm}{Theorem \ref{theoreme principal}*}

\theoremstyle{definition}
\newtheorem{remark}[X]{Remark}

\newcommand{\A}{\mathcal{A}}
\newcommand{\M}{\mathcal{M}}
\renewcommand{\a}{\mathbf{a}}
\newcommand{\K}{\mathbf{k}}
\newcommand{\h}{\mathbf{h}}
\newcommand{\LL}{\mathbf{L}}
\newcommand{\f}{\mathbf{f}}
\newcommand{\g}{\mathbf{g}}
\newcommand{\R}{\mathbf{R}}
\newcommand{\Ll}{\mathcal{L}}

\renewcommand{\S}{\mathcal S}

\renewcommand{\P}{\mathcal P}

\usepackage[latin1]{inputenc}
\usepackage[T1]{fontenc}
\usepackage{float}

\usepackage{mathrsfs}

\title{The influence of the first term of an arithmetic progression}
\author{Daniel Fiorilli}
\address{D\'epartement de math\'ematiques et de statistique \\ Universit\'e de Montr\'eal \\ CP 6128, succ.\ Centre-ville \\ Montr\'eal, QC \\ Canada H3C 3J7}
\email{fiorilli@dms.umontreal.ca}
\begin{document}

\begin{abstract}
The goal of this article is to study the discrepancy of the distribution of arithmetic sequences in arithmetic progressions. We will fix a sequence $\A=\{\a(n)\}_{n\geq 1}$ of non-negative real numbers in a certain class of arithmetic sequences. 
For a fixed integer $a\neq 0$, we will be interested in the behaviour of $\A$ over the arithmetic progressions $a \bmod q$, on average over $q$. Our main result is that for certain sequences of arithmetic interest, the value of $a$ has a significant influence on this distribution, even after removing the first term of the progressions. 
\end{abstract}
\maketitle

\tableofcontents

\section{Introduction}
The study of arithmetic sequences is a central problem in number theory. Undoubtedly, it is the sequence of prime numbers which has attracted the most attention amongst number theorists, 
leading to many theorems and conjectures. Other important sequences include sums of two squares, twin primes, divisor sequences and so on. In general, number theorists are interested in sequences with arithmetical content, and one can formally define wide classes of such sequences. Some phenomena occurring in the theory of prime numbers happen to be true for much wider classes of arithmetic sequences, such as the Bombieri-Vinogradov theorem for example (see \cite{motohashi}). Another example is the Granville-Soundararajan uncertainty principle (see \cite{uncertainty}).

We will fix an integer $a\neq 0$ and study the distribution of an arithmetic sequence $\A=\{\a(n)\}_{n\geq 1}$ in the progressions $a\bmod q$, on average over $q$. Under certain hypotheses, we will show how certain sequences remember the first term, that is how the value of $a$ can influence the distribution of $\A$ in the progressions $a\bmod q$. Examples of such sequences include the sequences of primes, sums of two squares (or more generally values of positive definite binary quadratic forms), prime $k$-tuples (conditionally) and integers without small prime factors. We will see that in each of these examples, values of $a$ which have the property that $\a(a)>0$ have a negative influence. More mysteriously, there are other values of $a$ having a negative influence, and it is not clear to me why these come up. 

The structure of the paper will be as follows. We first give concrete examples to highlight the phenomena we want to describe. Then we give a framework to study arithmetic sequences, as well as the definitions which will be needed. We also state the hypotheses on which our main theorems will depend. We then state the main results, leaving the proofs to the final chapters. 

\subsection{Acknowledgements}
I would like to thank my supervisor Andrew Granville for suggesting this generalization and for his help and advice in general. I would also like to thank my colleagues Farzad Aryan, Mohammad Bardestani, Dimitri Dias, Tristan Freiberg and Kevin Henriot for many fruitful conversations. Ce travail a été rendu possible grâce à des bourses doctorales du Conseil de Recherche en Sciences Naturelles et en Génie du Canada et de la Faculté des Études Supérieures et Postdoctorales de l'Université de Montréal.

\section{Examples}
\label{section exemples}

Before we state the general result, let us look at concrete examples. Throughout, $\A=\{\a(n)\}_{n\geq 1}$ will be a fixed sequence of non-negative real numbers and $a\neq 0$ will be a fixed integer, on which every error term can possibly depend. We will adopt the convention that for negative values of $a$, $\a(a):=0$ (and similarly for $\Lambda(a)$). Moreover, $M=M(x)$ will denote a function tending to infinity with $x$, and we will use $\sim$ as shorthand for $\sim_{M\rightarrow \infty}$ (similarly for $o(\cdot)=o_{M\rightarrow \infty}(\cdot)$). We define the following counting functions.
\begin{definition}
 \begin{equation*}
\A(x) :=\sum_{1\leq n\leq x} \a(n),  \hspace{1cm} \A_d(x):=\sum_{\substack{1 \leq n\leq x :\\ d \mid n}} \a(n),  \hspace{1cm} \A(x;q,a):=\sum_{\substack{1 \leq n\leq x \\ n\equiv a \bmod q}} \a(n).
\end{equation*}
\end{definition}

\subsection{Primes} The first example we give, which was studied more precisely in \cite{fiorilli}, is the sequence of prime numbers.

\begin{theorem}
Let $A>0$ be a fixed real number. We have for $M=M(x) \leq (\log x)^A$ that
\label{resultat premiers}
\begin{equation*}
 \frac 1{\frac{\phi(a)} a \frac xM} \sum_{\substack{q\leq \frac xM \\ (q,a)=1}} \left( \psi(x;q,a) - \Lambda(a) -\frac{\psi(x)}{\phi(q)}  \right)  \hspace{0.5cm} \text{ is } \hspace{0.5cm} \begin{cases}                                                                          
		\sim -\frac12 \log M  &\text{ if } a= \pm 1, \\
		\sim -\frac12 \log p   &\text{ if } a= \pm p^e, \\
		 \displaystyle = O_{\epsilon}\big( M^{-\frac{205}{538}+\epsilon} \big) &\text{ otherwise. }
   \end{cases}
\end{equation*} 
\end{theorem}

\subsection{Integers represented by a fixed binary quadratic form, with multiplicity}
The second example we consider is the sequence of integers which can be represented by a fixed binary quadratic form $Q(x,y)$ with integer coefficients, counted with multiplicity, that is $$\a(n):=\#\{(x,y)\in \mathbb Z_{\geq 0}^2:Q(x,y) = n\}.$$
We will define $r_d(n)$ to be the total number of distinct representations of $n$ by all of the inequivalent forms of discriminant $d$ (which is not to be confused with $\a(n)$). By distinct representations, we mean that we count the representations up to automorphisms of the forms. We also define the function
\begin{equation*} \rho_a(q) := \frac 1q \cdot \#\{ 1\leq x,y \leq q : Q(x,y) \equiv a \bmod q \}.\end{equation*}
\begin{theorem}
\label{resultat formes quad generales}
Suppose that $Q(x,y)=\alpha x^2+\beta xy + \gamma y^2$ is a fixed positive definite quadratic form (with integer coefficients) of discriminant $d:= \beta^2 - 4 \alpha \gamma <0$, with $d \equiv 1,5,9,12,13 \bmod 16$ (for simplicity). Fix an integer $a$ such that $(a,2d)=1$. We have for $M= M(x) \leq x^{\lambda}$, where $\lambda < \frac 1{12}$ is a fixed real number, that
\begin{equation}
\frac 1{x/M}\sum_{\substack{q\leq \frac xM }} \left( \A(x;q,a) - \a(a) -\frac{\rho_a(q)}{q}\A(x)  \right)  =  - C_Q \rho_a(4d) r_d(|a|)+O_{\epsilon}\left( \frac 1{M^{1/3 - \epsilon}} \right),
\label{equation resultat formes quad generales}
\end{equation}
with 
$$ C_Q:=\frac{A_{Q}}{2L(1,\chi_d)} \hspace{1cm} \left(= \frac{w_d \sqrt{|d|}}{4\pi h_{d}}A_{Q}\right),$$
where $A_{Q}$ is the area of the region $\{(x,y) \in \mathbb R_{\geq 0}^2 : Q(x,y) \leq 1\}$, $\chi_d:=\left( \frac{4d}{\cdot}\right)$,
$w_d$ is the number of units of $\mathbb Q(\sqrt d)$ and $h_{d}$ is its class number.
\end{theorem}

\begin{rem}
 The number $\rho_a(4d)$ is either zero or equal to $2^{\omega(2d)}$, $2^{\omega(2d)-2}$ or $3\cdot 2^{\omega(2d)-2}$, depending on $Q(x,y)$ ($\omega(n)$ denotes the number of prime factors of $n$). For this reason, if $\rho_a(4d)>0$, then it is independent of $a$.
\end{rem}

Therefore, there is no bias if $\rho_a(4d)=0$ or if $|a|$ cannot be represented by a form of discriminant $d$. 
However, if this is not the case, then the bias is proportional to the number of such representations.

\subsection{Sums of two squares, without multiplicity}
The next example is the sequence of integers which can be written as the sum of two squares, without multiplicity.
We define
$$\a(n):=\begin{cases}
               1 &\text{ if } n=\square+\square,\\
		0 &\text{ else.}
              \end{cases}
$$ 
For a fixed odd integer $a$, we define the multiplicative function $\g_a(q)$ on prime powers as follows. 
For $p\neq 2$ such that $p^f \parallel a$ with $f\geq 0$,
\begin{equation}\g_a(p^e) := \frac 1{p^e} \times \begin{cases}
                1 & \text{ if } p\equiv 1 \bmod 4 \\
		1 & \text{ if } p\equiv 3 \bmod 4, e\leq f, 2\mid e \\
		\frac 1p & \text{ if } p\equiv 3 \bmod 4, e\leq f, 2\nmid e \\
		1+ \frac 1p & \text{ if } p\equiv 3 \bmod 4, e>f, 2\mid f \\
		0 & \text{ if } p\equiv 3 \bmod 4, e>f, 2\nmid f.
               \end{cases}
\label{def de g_a pour poids 1}
\end{equation}
Moreover, $\g_a(2):=\frac 12$ and for $e\geq 2$, $\g_a(2^e):=\frac {1+(-1)^{\frac{a-1}2}}{2^{e+2}}$.
\begin{theorem}
 Fix an integer $a\equiv 1 \bmod 4$. We have for $1\leq M(x) \leq (\log x)^{\lambda}$, where $\lambda <1/5$ is a fixed real number, that  
\begin{multline}
\frac 1{x/{2M}}\sum_{\substack{\frac x{2M}< q\leq \frac xM }} \left( \A(x;q,a) - \a(a) -\g_a(q)\A(x)  \right)   \\
\sim -\left(\frac{\log M}{\log x}\right)^{\frac 12} \frac{(-4)^{-l_a-1}(2l_a+2)!}{(4l_a^2-1)(l_a+1)!\pi}  \prod_{\substack{p^f \parallel a: \\ p\equiv 3 \bmod 4, \\ f \text{odd}}} \frac{\log (p^{\frac{f+1}2})}{\log M},
\label{equation resultat dyadique poids 1}
\end{multline}
where $l_a:= \#\{ p^f \parallel a : p\equiv 3 \bmod 4, 2 \nmid f \}$ is the number primes dividing $a$ to an odd power which are congruent to $3$ modulo $4$.
\label{resultat dyadique poids 1}
\end{theorem}

\begin{rem}
The right hand side of \eqref{equation resultat dyadique poids 1} is $o((\log x)^{-1/2})$ iff $|a|$ cannot be written as the sum of two squares. Also, if $|a|=\square+\square$, then it is equal to
$-\frac 1{2\pi}\left(\frac{\log M}{\log x}\right)^{\frac 12}.$ Moreover, one can show that if $a\equiv 3 \bmod 4$, then the left hand side of \eqref{equation resultat dyadique poids 1} is always $o((\log x)^{-1/2})$.
\end{rem}

\subsection{Prime $k$-tuples}
The next example concerns prime $k$-tuples. Let $\mathcal H= \{ \Ll_1,...,\Ll_k\}$ be a $k$-tuple of distinct linear forms $\Ll_i(n)=a_i n + b_i$, with $a_i,b_i \in \mathbb Z$, $a_i \geq 1$, and define $$\P(n;\mathcal H):=\prod_{\Ll \in \mathcal H} \Ll(n).$$ We will suppose that $\mathcal H$ is admissible, that is for every prime $p$,
$$\nu_{\mathcal H} (p) := \# \{ x \bmod p : \P(x;\mathcal H) \equiv 0 \bmod p \} < p.$$
Define
$$ \a(n):= \prod_{\Ll \in \mathcal H} \Lambda(\Ll(n)) = \Lambda(a_1n+b_1) \Lambda(a_2n+b_2) \cdots \Lambda(a_k n+b_k).$$
The singular series associated to $\mathcal H$ is $$ \mathfrak S(\mathcal H) := \prod_p \left(1- \frac{\nu_{\mathcal H}(p)}{p} \right) \left( 1-\frac 1p\right)^{-k}.$$
Note that if $(\mathcal \P(a;\mathcal H),q)>1$, then $\A(x;q,a)$ is bounded. 
Fix $\delta>0$. The Hardy-Littlewood conjecture stipulates that there exists a function $\LL(x)$ tending to infinity with $x$ such that if $(\P(a;\mathcal H),q)=1$,
\begin{equation} \A(x) = \mathfrak S(\mathcal H) x + O\left(\frac x{\LL(x)^{2+2\delta}} \right).
\label{Hardy-Littlewood}
\end{equation}
Define
$$ \gamma(q):= \prod_{p\mid q} \left( 1- \frac {\nu_{\mathcal H} (p)}p \right).$$

\begin{theorem}
\label{resultat ktuplets}
 Assume that \eqref{Hardy-Littlewood} holds uniformly for all admissible $k$-tuples $\mathcal H$ such that $|a_i| \leq \LL(x)^{1+\delta}$ and $|b_i|=O(1)$. Then we have for $M=M(x)\leq \LL(x)$ that the average
\begin{multline*}
\frac 1{ \frac{\phi(\P(a;\mathcal H))}{\P(a;\mathcal H)}\frac{x}{2M}}\sum_{\substack{\frac x{2M}< q\leq \frac xM : \\(q,\P(a;\mathcal H))=1  }} \left( \A(x;q,a) - \a(a) - \frac{\A(x)}{q \gamma(q)}   \right)  \text{ is } \\
 \begin{cases} \displaystyle        
 \sim -\frac{(\log M)^{k-\omega(P(a;\mathcal H))}}{2(k-\omega(\P(a;\mathcal H)))!} \prod_{p\mid \P(a;\mathcal H)} \frac{p-\nu_{\mathcal H}(p)}{p-1}\log p  & \text{ if } \omega(\P(a;\mathcal H)) \leq k, \\
=O\left( \frac 1{M^{\delta_k}}\right) & \text{ otherwise, }
             \end{cases}
\end{multline*}
where $\delta_k>0$ is a positive real number depending on $k$, and $\omega(n)$ denotes the number of prime factors of $n$.
\end{theorem}

 In the case of twin primes, we have $\mathcal H= \{n,n+2\}$, so $\P(a,\mathcal H) = a(a+2)$, and the function $\nu_{\mathcal H}$ is given by $\nu_{\mathcal H}(2)=1$ and $\nu_{\mathcal H}(p)=2$ for odd $p$.
We get that the average is
$$
 \begin{cases}
 \sim                  -\frac{ (\log M)^2}{4} & \text{ if } a=-1  \\
\sim			-\frac {\log 3}{4}  \log M   	& \text{ if } a=1,-3 \\
\sim-\frac{\log 2}2  \log M 	& \text{ if } a=2,-4 \\
	\sim			-\frac{\log p \log q}{2}\frac{p-\nu_{\mathcal H}(p)}{p-1}\frac{q-\nu_{\mathcal H}(q)}{q-1}	& \text{ if } a(a+2) = \pm p^e q^f \\
		O\left( \frac 1{M^{\delta_2}}\right) &\text{ if } \omega(a(a+2)) \geq 3.
                  \end{cases}
$$

\subsection{Integers free of small prime factors}

For $y=y(x)$ a function of $x$, define 
\begin{align*} \a_y(n):=\begin{cases}
                   1 &\text{ if } p\mid n \Rightarrow p\geq y\\
 0 & \text{ else,}
                  \end{cases}
\hspace{2cm} \A(x,y)&:= \sum_{n\leq x} \a_y(n), \\
\gamma_y(q):=\prod_{\substack{p\mid q \\ p<y}}\left(1-\frac 1p\right), \hspace{2.8cm} \A(x,y;q,a)&:= \sum_{\substack{n\leq x \\ n\equiv a\bmod q}} \a_y(n).
\end{align*}

\begin{theorem}
 \label{resultat nombres rough}
Fix $a\neq 0$, $\delta>0$ and $M=M(x)\leq (\log x)^{1-\delta}$. If
\begin{equation*}\nu_y(a,M) := \frac{1}{\frac {x}{2M} \frac{\phi(a)}{a}}\sum_{\substack{\frac x{2M} < q \leq \frac xM \\ (q,a)=1}} \left( \A(x,y;q,a)-\a_y(a)- \frac {\A(x,y)}{q\gamma_y(q)} \right),
\end{equation*}
then for $y\leq e^{(\log M)^{\frac 12-\delta}}$ with $y\rightarrow \infty$,
$$  \nu_y(a,M)=\begin{cases}
                -\frac 12+o(1) &\text{ if } a=\pm 1 \\
		o(1)& \text{ otherwise,}
               \end{cases}
$$
and for $(\log x)^{\log \log \log x} \leq y \leq \sqrt x,$
$$  \nu_y(a,M)=\frac{\A(x,y)}{x} \times \begin{cases}
                (1+o(1)) \log M &\text{ if } a=\pm 1 \\
\log p +o(1)&\text{ if } a=\pm p^k \\
		o(1) & \text{ otherwise.}
               \end{cases}
$$
(We have no result in the intermediate range.)
\end{theorem}

\begin{rem}
For $x$ large enough, $\a_y(a)=0$ unless $a=\pm 1$.
\end{rem}

\section{Definitions and Hypotheses}

\subsection{Arithmetic sequences}\label{cadre}

The goal of this section is to give a framework to study arithmetic sequences. This discussion is modeled on that in \cite{uncertainty}.

We wish to study the sequence $\A=\{\a(n)\}_{n\geq 1}$ in arithmetic progressions, therefore one of our goals will be to prove the existence of a multiplicative function $\g_a(q)$ such that 
$$ \A(x;q,a)\sim \g_a(q) \A(x),$$
whenever $\g_a(q)\neq 0$. 
Let us give a heuristic way to do this with the help of an auxiliary multiplicative function $\h(d)$. 
First, denote by $\S$ a finite set of "bad primes", which are inherent to the sequence $\A$. We will assume that $\A$ is well distributed in the progressions $0\bmod d$, that is there exists a multiplicative function $\h(d)$ such that for $(d,\S)=1$,
\begin{equation*}
 \A_d(x) \approx \frac{\h(d)}{d} \A(x).
\end{equation*}
The fact that $\h(d)$ is multiplicative can be rephrased as "the events that $\a(n)$ is divisible by coprime integers are independent".
Let us also assume that
\begin{equation*}
 \A(x;q,a) \approx \frac 1{\phi(q/(q,a))} \sum_{\substack{n\leq x : \\ (q,n)=(q,a)}} \a(n), 
\end{equation*}
that is the sum is equally partitioned amongst the $\phi(q/(q,a))$ arithmetic progressions $b\bmod q$ with $(b,q)=(a,q)$. We then compute
\begin{multline*}
 \A(x;q,a) \approx \frac 1{\phi(q/(q,a))} \sum_{\substack{n\leq x \\ (q,n)=(q,a)}} \a(n) = \frac 1{\phi(q/(q,a))} \sum_{d\mid \frac q{(q,a)}} \mu(d) \A_{(q,a)d}(x) \\ \approx  \A(x)\frac 1{\phi(q/(q,a))} \sum_{d\mid \frac q{(q,a)}} \mu(d) \frac{\h((q,a)d)}{(q,a)d} = \g_a(q) \A(x),
\end{multline*}
where $$\g_a(q)=\g_{(a,q)}(q):=\frac 1{\phi(q/(q,a))} \sum_{d\mid \frac q{(q,a)}} \mu(d) \frac{\h((q,a)d)}{(q,a)d}$$ is a multiplicative function of $q$ which depends on $(q,a)$ (rather than depending on $a$). We have thus expressed the multiplicative function $\g_a(q)$ in terms of $\h(d)$. More explicitly, we have, when $p^f \parallel a$ (with $(pa,\S)=1$), that
\begin{equation}
\label{definition de g avec h}
\g_a(p^e)=\begin{cases}
            \frac{\h(p^e)}{p^e} & \text{ if } e\leq f \\
		\frac{1}{\phi(p^e)} \left( \h(p^f)-\frac{\h(p^{f+1})}{p} \right) & \text{ if } e > f.
           \end{cases}
\end{equation}
In particular, if $p\nmid a$, $$\g_a(p^e)= \frac 1{\phi(p^e)}\left(1-\frac{\h(p)}p\right).$$
Another way to write this is
\begin{equation}
\label{approx sur A(x;q,a)}
 \A(x;q,a) \approx \frac{\f_a(q)}{q\gamma(q)} \A(x),
\end{equation}
where $$\gamma(q):= \frac{\phi(q)}q \prod_{p\mid q}\left(1-\frac{\h(p)}p\right)^{-1} = \prod_{p\mid q}\frac{1-1/p}{1-\h(p)/p},$$
and $\f_a(q)$ is a multiplicative function defined by $\f_a(q):= \g_a(q)q\gamma(q)$. Note that for $(a,q)=1$, $\f_a(q)=1$.

\subsection{Hypotheses}

In the following, $\delta>0$ will denote a (small) fixed real number which will change from one statement to another. We will also fix an integer $a\neq 0$  with the property that $(a,\S)=1$, where $\S$ is a finite set of bad primes. The function $\LL: [0,\infty) \rightarrow [1,\infty)$ will be a given increasing smooth function such that $\LL(x) \rightarrow \infty$ as $x \rightarrow\infty$ (think of $\LL(x)$ as a power of $\log x$). We now assume the existence of a multiplicative function $\f_a(q)=\f_{(a,q)}(q)$, depending on $(a,q)$, and of $\gamma(q)\neq 0$, which is independent of $a$ (as in Section \ref{définitions}), such that for any fixed $a\neq 0$ and $q\geq 1$,
$$ \A(x;q,a) \sim \frac{\f_a(q)}{q\gamma(q)} \A(x)$$
whenever $\f_a(q)\neq 0$. To simplify the notation, we will also assume the existence of a multiplicative function $\h(d)$ such that \eqref{definition de g avec h} holds (for $(qa,\S)=1$). 

\label{hypothèses}

\begin{hypothesis}
\label{bombieri-vinogradov}
There exists a positive increasing function $\R(x)$ (think of $\R(x)$ as a small power of $x$), with $ \LL(x)^{1+\delta} \leq \R(x) \leq \sqrt x$, such that 

$$ \sum_{q\leq 2\R(x)} \max_{y\leq x} \left| \A(y;q,a)-\frac{\f_a(q)}{q\gamma(q)} \A(y) \right|\ll \frac{\A(x)}{ \LL(x)^{1+\delta}}.$$

\end{hypothesis}

We will see later that if we use dyadic intervals, we can replace Hypothesis \ref{bombieri-vinogradov} by a weaker hypothesis.
\begin{hyp}
We have
$$ \sum_{q\leq 2\LL(x)} \max_{y\leq x} \left| \A(y;q,a)-\frac{\f_a(q)}{q\gamma(q)} \A(y) \right|\ll \frac{\A(x)}{ \LL(x)^{1+\delta}}.$$

\end{hyp}

\begin{hypothesis}
\label{taille de A}
For any $z=z(x)$ in the range $ \frac 1{\LL(x)} \leq  z(x)  \leq 1+\frac{|a|}{x}$, we have
 \begin{equation*}
 \frac{\A(zx)}{\A(x)} = z +O\left(\frac 1 {\LL(x)^{1+\delta}}\right).
\end{equation*}
Moreover, for $n\leq x$, we have the following bound: 
$$\a(n) \ll \frac {\A(x)} {\LL(x)^{1+\delta}}.$$
\end{hypothesis}
The next hypothesis is somewhat more specific to our analysis than the ones above, and it will allow us to use the analytic theory of zeta functions.
\begin{hypothesis}
\label{hypothèse sur h}
There exists a real number $\K \geq 0$ such that the sum 
$$ \sum_{p\notin \S} \frac{\h(p)-\K}{p}$$
is convergent. More generally, for any real number $t$ and integer $n\geq 1$, we have 
$$\sum_{\substack{p\leq x \\ p\notin \S}} \frac{\h(p)-\K}{p^{1+it}}  \leq ( 1/2-\delta) \log (|t|+2)+O(1),$$
$$ \sum_{\substack{p\leq x \\ p\notin \S}} \frac{(\h(p)-\K)\log ^n p}{p^{1+it}} \ll_{n,\epsilon} (|t|+2)^{\epsilon}. $$

Finally, $\h(p)<p$ and for any $\epsilon>0$,
$$ \h (d)\ll_{\epsilon} d^{\epsilon}. $$
\end{hypothesis} 

The final hypothesis will be useful when studying the full interval $1\leq q \leq \frac xM$ rather than a dyadic one. It is not known for all the sequences we considered in Section \ref{section exemples}; for this reason we used dyadic intervals in theorems \ref{resultat dyadique poids 1}, \ref{resultat ktuplets} and \ref{resultat nombres rough}.

\begin{hypothesis}
\label{bombieri friedlander iwaniec}
With the same $\R(x)$ as in Hypothesis \ref{bombieri-vinogradov}, we have
$$\sum_{q\leq \frac x{\R(x)}} \left( \A^*(x;q,a)-\frac{\f_a(q)}{q\gamma(q)} \A(x) \right) \ll \frac {\A(x)} {\LL(x)^{1+\delta}},$$
where $\A^*(x;q,a)$ is defined as in \eqref{def de A*}.
\end{hypothesis}

\subsection{The formula for the average}
In this section we give a formula for the "average" $\mu_{\K} (a,M)$ which will appear in theorems \ref{theoreme principal} and \ref{theoreme principal}*. The formula is rather complicated in its general form, however in concrete examples it can be seen that it reflects the nature of the sequence $\A$. 

\label{définitions}
\begin{definition}
 $$\omega_{\h}(a):=\#\{p^f\parallel a \text{ with } f\geq 1 :\h(p^f)=\h(p^{f+1})/p \}.$$
\end{definition}

\begin{definition}
 \label{définition de mu_k}

Assume Hypothesis \ref{hypothèse sur h} and suppose that $\S=\emptyset$ (for simplicity). For an integer $a\neq 0$ and a real number $\K\geq 0$, we define

\begin{multline}
 \mu_{\K} (a,M):=  -\frac 12 \frac{(\log M)^{1-\K-\omega_{\h}(a)}}{\Gamma(2-\K -\omega_{\h}(a))} \prod_{\substack{p^f \parallel a  :\\ \h(p^f)=\frac{\h(p^{f+1})}p, \\ f\geq 0}} \frac{1+\h(p)+...+\h(p^{f})}{\left( 1-1/p\right)^{\K-1}} \log p \hspace{0.2cm}   \\
 \times\prod_{\substack{p^f \parallel a : \\ \h(p^f)\neq \frac{\h(p^{f+1})}p, \\ f\geq 0}} \frac{\h(p^f)-\h(p^{f+1})/p}{(1-1/p)^{\K}}. 
\label{def generale de mu}
\end{multline}

\end{definition}

\begin{rem}
 The first product on the right hand side of \eqref{def generale de mu} is a finite product, since $a$ is fixed and $\h(p)<p$ for all $p$. The second product is convergent, since for $p\nmid a$ we have 
$\h(p^f)-\h(p^{f+1})/p = 1-\h(p)/p \approx 1- \K/p$. Of course both these statements rely on the assumption of Hypothesis \ref{hypothèse sur h}.
\end{rem}

\begin{rem}
One sees that for integer values of $\K$, $\mu_{\K} (a,M)=0$ iff $\omega_{\h}(a) \geq 2-\K$, by the location of the poles of $\Gamma(s)$. Moreover, since these are the only poles, we have $\mu_{\K}(a,M)\neq 0$ whenever $\K \notin \mathbb Z$.
\end{rem}

\begin{rem}
 If $\S \neq \emptyset$, we can still give a formula for $\mu_{\K} (a,M)$, assuming we understand well $\g_a(p^e)$ with $p\in \S$. However, this would complicate the already lengthy definition of $\mu_{\K} (a,M)$, so we only give individual descriptions in the examples. 
\end{rem}

\section{Main result}
The main result of the paper is a formula for the average value of the discrepancy $\A(x;q,a)-\frac{\f_a(q)}{q\gamma(q)} \A(x)$, summed over $1\leq q\leq Q$, with $Q$ large enough in terms of $x$.
\begin{theorem}
\label{theoreme principal}
Assume that hypotheses \ref{bombieri-vinogradov}, \ref{taille de A}, \ref{hypothèse sur h} and \ref{bombieri friedlander iwaniec} hold with $\S=\emptyset$ (for simplicity) and the function $\LL(x)$. Fix an integer $a\neq 0$ and let $M=M(x)$ be a function of $x$ such that $1\leq M(x)\leq \LL(x)$. We have

\begin{equation}
\sum_{q\leq \frac xM} \left( \A(x;q,a) - \a(a)-\frac{\f_a(q)}{q\gamma(q)} \A(x) \right)  = \frac{\A(x)}{M} \left(\mu_{\K}(a,M)(1+o(1)) +O_A \left( \frac 1 {\log^ A M}\right) \right),
\label{equation thm principal}
\end{equation}
where $\a(a)$ is the first term of $\A(x;q,a)$ for positive $a$, and whenever $a$ is negative, we set $\a(a)=0$.
\end{theorem}
We also give a dyadic version, which assumes a weaker form of Hypothesis \ref{bombieri-vinogradov}, and does not assume Hypothesis \ref{bombieri friedlander iwaniec} at all.

\begin{thm}
Assume that hypotheses \ref{bombieri-vinogradov}*, \ref{taille de A} and \ref{hypothèse sur h} hold with $\S=\emptyset$ (for simplicity) and the function $\LL(x)$. Fix an integer $a\neq 0$ and let $M=M(x)$ be a function of $x$ such that $1\leq M(x)\leq \LL(x)$. We have
\begin{equation}
\sum_{ \frac{x}{2M} < q\leq \frac xM} \left( \A(x;q,a) - \a(a)-\frac{\f_a(q)}{q\gamma(q)} \A(x) \right)  = \frac{\A(x)}{2M} \left(\mu_{\K}(a,M)(1+o(1)) +O_A \left( \frac 1 {\log^ A M}\right) \right).
\label{equation thm principal dyadique}
\end{equation}
\end{thm}

\begin{remark}
As we have seen in the examples of Section \ref{section exemples}, theorems \ref{theoreme principal} and \ref{theoreme principal}* easily generalize to arbitrary (given) sets $\S\neq \emptyset$, as long as we understand $\g_a(p^e)$ for each $p\in \S$.
\end{remark}

\begin{remark}
 If $\mu_{\K}(a,M)\neq 0$, then theorems \ref{theoreme principal} and \ref{theoreme principal}* give asymptotics for the sum on the left hand side.
\end{remark}

\begin{remark}

\label{remarque k=0}
Suppose that $\K=0$ (e.g. when $\A$ is the sequence of primes).

If $\omega_{\h}(a)\geq 2$, then $\mu_0(a,M)=0$.

If $\omega_{\h}(a) =1$, so there is a unique $p_0^{f_0}\parallel a$, $f_0\geq 1$, such that $\h(p_0^{f_0})=\h(p_0^{f_0+1})/p_0 $, then
 $$ \mu_0(a,M)=-\frac12 \left(1-\frac 1{p_0} \right) (1+\h(p_0)+...+\h(p_0^{f_0})) \log p_0 \prod_{\substack{p^f \parallel a \\ f\geq 0 \\ p\neq p_0}} \left(\h(p^f)-\h(p^{f+1})/p\right). $$

If $\omega_{\h}(a)=0$, then

\begin{equation*}
 \mu_0(a,M)=-\frac {\log M}{2}\prod_{\substack{p^f\parallel a \\ f\geq 0}} (\h(p^f)-\h(p^{f+1})/p) .
\end{equation*}
\end{remark}

\begin{remark}
\label{remarque k=1}
Suppose that $\K=1$ (e.g. when $\A$ is the sequence of integers which can be written as the sum of two squares, counted with multiplicity). Then
$$\mu_1(a,M)=-\frac 1{2}\prod_{\substack{p^f \parallel a\\ f\geq 0}} \frac{\h(p^f)-\h(p^{f+1})/p}{1-1/p} .$$
\end{remark}

\begin{remark}
\label{remarque k>=2}
Suppose that $\K$ is an integer $\geq 2$ (e.g. when $\A$ is the sequence of integers of the form $(m+c_1)(m+c_2)\cdots(m+c_{\K})$, where the $c_i$ are distinct integers). Then $\mu_1(a,M)=0.$
\end{remark}

\section{Proof of the main result}
The goal of this section is to prove theorems \ref{theoreme principal} and \ref{theoreme principal}*.

\subsection{An estimate for the main sum}

In this section, we will assume that $\S=\emptyset$ for simplicity. Again, the results easily generalize to $\S \neq \emptyset$.

\begin{proposition}
 \label{proposition evaluation de la somme arithmetique}
Assume Hypothesis \ref{hypothèse sur h}. Let $M=M(x)$ and $\R=\R(x)$ be two positive functions of $x$ such that $M(x)^{1+\delta} \leq \R(x) \leq \sqrt x$ for a fixed $\delta>0$. We have
\begin{multline*}
 \sum_{1\leq r \leq  \R}\frac{\f_a(r)}{r\gamma(r)}\left( 1-\frac{r}{\R} \right)-\sum_{1\leq r \leq  M}\frac{\f_a(r)}{r\gamma(r)}\left(1- \frac{r}{M} \right) - \sum_{\frac x{\R}< q \leq  \frac x M}\frac{\f_a(q)}{q\gamma(q)}  \\
= \frac{\mu_k(a,M)}{M}\left( 1+O\left( \frac{\log\log M}{\log M}\right) \right)
+ O_A\left( \frac 1 {M\log^ A M}\right).
\end{multline*}
\end{proposition}
The proof of Proposition \ref{proposition evaluation de la somme arithmetique} will require several lemmas.

\begin{lemma}
 \label{borne sur g}
With $\f_a(n)$ and $\gamma(n)$ defined as in Section \ref{cadre}, we have $$\frac{\f_a(n)}{n\gamma(n)} \ll \frac 1 {\phi(n)}.$$
\end{lemma}
\begin{proof}
 
By definition, 

\begin{align*}
\frac{\f_a(n)}{n\gamma(n)}=\g_a(n) &= \prod_{p^e \parallel n} \g_a(p^e) \ll_{a,\S} \prod_{\substack{p^e \parallel n \\ p\nmid a, p\notin \S}} \g_a(p^e) \\
&= \prod_{\substack{p^e \parallel n \\ p\nmid a , p\notin \S}} \frac{1}{\phi(p^e)} \left( 1-\frac{\h(p)}p\right)  \\
&\leq \prod_{\substack{p^e \parallel n \\ p\nmid a, p\notin \S}} \frac{1}{\phi(p^e)} \ll_{a,\S} \frac 1 {\phi(n)}.
 \end{align*}
\end{proof}

\begin{lemma}
Assume Hypothesis \ref{hypothèse sur h}. Let $h:[0,\infty) \rightarrow \mathbb [0,\infty)$ be a piecewise continuous function supported on $[0,1]$, taking a value halfway between the limit values at discontinuities, and suppose the integral
$$ \M h(s) := \int_0^1 h(x) x^{s-1}dx$$
converges absolutely for $\Re (s)>0$. Then,
\label{transformée de Mellin}
\begin{equation}
\label{equation transformée de Mellin}
\sum_{n\leq M} \frac{\f_a(n)}{n\gamma(n)} h\left( \frac nM\right)= \frac 1{2\pi i} \int_{(1)} \mathfrak{S}_2(s) \zeta(s+1) \zeta(s+2)^{1-\K} Z_5(s) \M h(s) M^s ds,
\end{equation}
where

\begin{multline*}
\mathfrak S_2 (s):= \prod_{\substack{p^f \parallel a \\ f\geq 1}} \Bigg[ \left( 1+\frac{\h(p)}{p^{s+1}}+...+\frac{\h(p^{f})}{p^{f(s+1)}} \right) \left(1-\frac 1{p^{s+1}} \right) \\
 + \frac{ \h(p^f)-\h(p^{f+1})/p}{1-1/p } \frac 1 {p^{(f+1)(s+1)}} \Bigg] \left(1-\frac 1 {p^{s+2}} \right)^{1-\K},
\end{multline*}
\begin{equation}Z_5(s):=\prod_{p \nmid a} \left( 1+\frac 1 {p^{s+1}} \left( \frac 1 {\gamma(p)}-1 \right) \right)\left(1-\frac{1}{p^{s+2}}\right)^{1-\K}. \label{definition de Z_5}
\end{equation}
 Moreover, $\mathfrak S_2 (s)$ is holomorphic in $\mathbb C \setminus \{ -2 \}$ and $Z_5(s)$ is holomorphic for $\Re s > -1$.

\end{lemma}
\begin{proof}
Define
$$ Z_{\A}(s):=\sum_{n=1}^{\infty} \frac{\g_a(n)}{n^s} = \prod_p \left(1+\frac{\g_a(p)}{p^s}+\frac{\g_a(p^2)}{p^{2s}}+... \right).$$
A standard computation using the definition of $\g_a(n)$ (see \eqref{definition de g avec h}) yields that
\begin{equation*}Z_{\A}(s) = \mathfrak S_2 (s) \zeta(s+1) \zeta(s+2)^{1-\K}  Z_5(s). \end{equation*}
The function $\mathfrak S_2 (s)$ is clearly holomorphic in $\mathbb C \setminus \{ -2 \}$, and the fact that $Z_5(s)$ is holomorphic for $\Re s > -1$ follows from Hypothesis \ref{hypothèse sur h}.
Now, Mellin inversion gives that
$$ h\left(\frac nM\right) = \frac 1 {2\pi i } \int_{(1)} \frac{M^s}{n^s} \M h(s) ds. $$
Multiplying by $\frac{\f_a(n)}{n\gamma(n)}$ and summing over $n$ yields the result.
\end{proof}

\subsubsection{Properties of the Dirichlet series}

\begin{lemma}
\label{régularité de Z_5 en -1}
Assume Hypothesis \ref{hypothèse sur h}. We have 
$$ Z_5(s)=Z_5(-1) +O( |s+1|) $$
in the region $|s+1|\leq 3$, with $\Re s > -1$.
\end{lemma}
\begin{proof}
We will show that
$$ \log \frac{Z_5(s)}{Z_5(-1)} \ll |s+1|,$$
from which the lemma clearly follows. Let $s$ be a complex number with $\Re s >-1$. We compute

\begin{align*}
 \log \frac{Z_5(s)}{Z_5(-1)} &= \sum_{p\nmid a} \log \left( \frac{1+\frac 1{p^{s+1}}(\frac 1 {\gamma(p)}-1)}{ \frac 1 {\gamma(p)}} \cdot \frac{\left(1-\frac 1{p^{s+2}} \right)^{1-\K}}{\left(1-\frac 1{p} \right)^{1-\K}}\right) \\
&= \sum_{p\nmid a} \Bigg[  \log \left( 1-(1-\gamma(p)) \left( 1-\frac 1 {p^{s+1}}  \right) \right)  \\
& \hspace{4cm}+(1-\K)  \log \left( 1+\frac 1{p-1} \left( 1- \frac 1{p^{s+1}} \right)\right)  \Bigg] \\
&=\sum_{p\nmid a} \left[\frac{\h(p)-1}{p-\h(p)} \left( 1-\frac 1 {p^{s+1}}  \right) +\frac {1-\K}{p-1} \left( 1- \frac 1{p^{s+1}} \right)\right]  \\
& \hspace{5cm}+ O_{\epsilon}\left(|s+1|^2 \sum_p \frac {\log^ 2 p}{p^{2-\epsilon}}\right) 
\end{align*}
\begin{align}
\label{equation régularité en -1}
& =  \sum_{p\nmid a} \left( \frac{\h(p)-1}{p-\h(p)} + \frac {1-\K}{p-1} \right) \left( 1-\frac 1 {p^{s+1}}  \right) +O\left(  |s+1|^2 \right) \notag \\
& =  \sum_{p} \left( \frac{\h(p)-1}{p-\h(p)} + \frac {1-\K}{p-1} \right) \left( 1-\frac 1 {p^{s+1}}  \right) +O\left(  |s+1| \right).
\end{align}

Note that by Hypothesis \ref{hypothèse sur h}, the series
$$ \sum_p \left( \frac{\h(p)-1}{p-\h(p)} + \frac {1-\K}{p-1} \right) = \sum_p \frac{\h(p)-\K}p + O(1) $$
converges. Moreover, summation by parts yields the following estimate:
$$ S(t):=\sum_{p\leq t} \left( \frac{\h(p)-1}{p-\h(p)} + \frac {1-\K}{p-1} \right) = S(\infty) + O\left(\frac 1{\log^2 (t+2)} \right). $$
We then get that
\begin{align*}  
\sum_{p\leq T} &\left( \frac{\h(p)-1}{p-\h(p)} + \frac {1-\K}{p-1} \right) \left( 1-\frac 1 {p^{s+1}}  \right) =  \int_1^T \left( 1-\frac 1 {t^{s+1}}  \right) dS(t) \\
&= \left( 1-\frac 1 {t^{s+1}}  \right) S(t) \Bigg|_1^{T} - (s+1)\int_1^{T} \frac{S(t)}{t^{s+2}} dt \\
&= \left( 1-\frac 1 {T^{s+1}}  \right)\left(S(\infty) + O\left(\frac 1{\log^2 T} \right) \right) - (s+1) \int_1^{T} \frac{S(\infty)}{t^{s+2}} dt  \\
 & \hspace{4cm}+ O\left(|s+1| \int_1^{T} \frac{dt}{t\log ^2(t+2)} \right) \\
&= S(\infty) \left( 1-\frac 1 {T^{s+1}}  \right) +O\left(\frac 1{\log^2 T}   \right) + \frac {S(\infty)} {t^{s+1}} \Bigg| _1^ T + O\left( |s+1|\right) \\
&=O\left( \frac 1{\log^2 T} + |s+1| \right).
\end{align*}
Taking $T\rightarrow \infty$ yields that \eqref{equation régularité en -1} is $\ll |s+1|$.
\end{proof}

\begin{lemma}
\label{dérivées et dérivées logarithmiques}
Let $f(s)$ be a holomorphic function over a domain $\mathcal D$. We have that $\frac{f^{(n)}}{f} (s)$ is a polynomial in the variables 
$  \left(\frac{f'(s)}{f(s)}\right)^{(0)},\left(\frac{f'(s)}{f(s)}\right)^{(1)},...,\left(\frac{f'(s)}{f(s)}\right)^{(n-1)}$, with integer coefficients.
\end{lemma}
\begin{proof}
The proof goes by induction, using the identity
$$ \frac{f^{(n)}}{f} = \left( \frac{f^{(n-1)}}f \right)'  + \frac{f^{(n-1)}}{f} \frac{f'}{f}.$$

\end{proof}

\begin{lemma}
\label{bornes sur Z_5}
Assume Hypothesis \ref{hypothèse sur h}. Let $Z_5(s)$ be defined as in \eqref{definition de Z_5} and let $n\geq 0$. Then there exists $\delta>0$ such that, uniformly in the region $-1<\sigma < -\frac 12$ and $t\in \mathbb R$, we have
\begin{equation} Z_5^{(n)}(\sigma+it) \ll_n (|t|+2)^{1/2-\delta} . \label{bornes sur dérivées de Z_5} \end{equation}
\end{lemma}
\begin{proof}

First write $Z_5(s)=Z_3(s)Z_4(s)$, where
\begin{align*}
 Z_3(s):= \prod_{p\nmid a} \left( 1+\frac 1 {p^{s+1}} \left( \frac 1 {\gamma(p)}-1 \right) \right) \left(1-\frac{1-\K}{p^{s+2}}\right),
\end{align*}
\begin{align*}
 Z_4(s):= \prod_{p\nmid a}\left(1-\frac{1}{p^{s+2}} \right)^{1-\K} \left(1-\frac{1-\K}{p^{s+2}}\right)^{-1}.
\end{align*}
The function $Z_4(s)$ is uniformly bounded in the region $\Re s \geq -1 $, since the Eulerian product converges absolutely. As for $Z_3(s)$, we have for $-1<\sigma < -\frac 12$ that
$$\log Z_3(\sigma+it) =  \log \prod_{p\nmid a} \left( 1+\frac 1{p^{\sigma+1+it}} \frac{\K-\h(p)}{p} \right) +O(1). $$
Hypothesis \ref{hypothèse sur h} gives
$$ S(x,t):=\sum_{p\leq x} \frac{\K-\h(p)}{p^{1+it}} \leq (1/2-\delta) \log(|t|+2)+O(1). $$
Thus, 
\begin{align*}
\log \prod_{p\nmid a} \left( 1+\frac 1{p^{\sigma+1+it}} \frac{\K-\h(p)}{p} \right)  &= \sum_{p\nmid a} \frac 1{p^{\sigma+1}} \frac{\K-\h(p)}{p^{1+it}} + O(1) \\
&=\int_1^{\infty} \frac{dS(x,t)}{x^{\sigma+1}}+O(1) \\
&= \frac{S(x,t)}{x^{\sigma+1}}\Big|_1^{\infty} + (\sigma+1) \int_1^{\infty} \frac{S(x,t)}{x^{\sigma+2}} dx +O(1) \\
&\leq (1/2-\delta) \log(|t|+2)   \int_1^{\infty} \frac{\sigma+1}{x^{\sigma+2}} dx +O(1) \\
&= (1/2-\delta) \log(|t|+2)+O(1),
\end{align*}
which proves \eqref{bornes sur dérivées de Z_5} for $n=0$. The bound
$$ \sum_{p\leq x} \frac{(\K-\h(p)) \log^m p}{p^{1+it}} \ll_{\epsilon} (|t|+2)^{\epsilon} $$
gives
\begin{equation}
\label{equation borne sur Z_5}
 \left(\frac{Z_5'(\sigma+it)}{Z_5(\sigma+it)}\right)^{(m)} \ll_{\epsilon} (|t|+2)^{\epsilon} 
\end{equation}
for $m \geq 0.$ We finish the proof of \eqref{bornes sur dérivées de Z_5} for $n\geq 1$ by applying Lemma \ref{dérivées et dérivées logarithmiques}.

\end{proof}

\begin{lemma}
\label{bornes sur zeta}
 We have for $|\sigma+it-1|>\frac 1 {10}$ that $$ \zeta(\sigma+it)\ll_{\epsilon} (|t|+2)^{\mu(\sigma)+\epsilon}, $$ where
\begin{equation*}
\mu(\sigma)=\begin{cases}
             1/2-\sigma &\text{ if } \sigma \leq 0 \\
1/2-2\sigma/3 &\text{ if } 0 \leq \sigma \leq 1/2 \\
1/3-\sigma/3 &\text{ if } 1/2\leq  \sigma \leq 1 \\
0  &\text{ if } \sigma \geq 1. \\
            \end{cases}
\end{equation*}
Moreover, these bounds are uniform for $\sigma$ contained in any compact subset of $\mathbb R$.
\end{lemma}
\begin{proof}
 See Section II.3.4 of \cite{tenenbaum}, in particular (II.3.13) and Theorem 3.8. By studying the proof of the Phragment-Lindelöf principle (see Chapter 9 of \cite{edwards} for instance), we see that the bounds we get are uniform in $\sigma$.
\end{proof}

\begin{lemma}
\label{borne sur dérivées de Z}
Assume Hypothesis \ref{hypothèse sur h}. Let 
$$Z(s):=   \frac{\mathfrak{S}_2(s) \zeta(s+1) \zeta(s+2)^{1-\K} Z_5(s)}{s(s+1)},$$
with $\mathfrak{S}_2(s)$ and $Z_5(s)$ defined as in Lemma \ref{transformée de Mellin}. There exists $\delta>0$ such that uniformly for $|t|\geq 2$ and $-1<\sigma<-\frac 12$,
$$ Z^{(n)}(\sigma+it) \ll_n \frac{1}{|t|^{1+\delta}}.  $$

\end{lemma}
\begin{proof}
Define  $$Z_6(s):=  \frac{\mathfrak{S}_2(s) \zeta(s+2)^{1-\K} Z_5(s)}{s(s+1)}.$$
Write $s=\sigma+it$, with $-1<\sigma<-\frac 12$ and $|t|\geq 2$. We have for $m \geq 0$ that
\begin{equation*}
 \left( \frac{Z_6'(s)}{Z_6(s)}\right)^{(m)} = \left( \frac{\mathfrak{S}'_2(s)}{\mathfrak{S}_2(s)}\right)^{(m)}
+(1-\K)\left( \frac{\zeta'(s+2)}{\zeta(s+2)}\right)^{(m)}
+\left( \frac{Z_5'(s)}{Z_5(s)}\right)^{(m)}-\left( \frac{2s+1}{s(s+1)}\right)^{(m)}.
\end{equation*}
We compute that
\begin{equation*}
 \left( \frac{\mathfrak{S}'_2(s)}{\mathfrak{S}_2(s)}\right)^{(m)} \ll_{m} 1, \hspace{1cm} \left( \frac{2s+1}{s(s+1)}\right)^{(m)} \ll_m 1, \hspace{1cm}\left( \frac{Z_5'(s)}{Z_5(s)}\right)^{(m)} \ll_{m,\epsilon} |t|^{\epsilon}.
\end{equation*}
(The first bound is clear, the second follows from the fact that $|t|\geq 2$ and the third comes from \eqref{equation borne sur Z_5}.) Applying Cauchy's formula for the derivatives as in Corollaire II.3.10 of \cite{tenenbaum} and then using the bound (II.3.55) of \cite{tenenbaum} yields
\begin{equation*}
 \left( \frac{\zeta'(s+2)}{\zeta(s+2)}\right)^{(m)} \ll_m \log^{m+1}(|t|).
\end{equation*}
Using Lemma \ref{dérivées et dérivées logarithmiques},
$$ Z_6^{(m)}(s) \ll_{\epsilon,m} |Z_6(s)| |t|^{\epsilon}$$
for $m\geq 0$. We now use Lemma \ref{bornes sur Z_5} to bound $|Z_5(s)|$, which gives
$$ Z_6^{(m)}(s) \ll_{m} |\zeta(s+2)^{1-\K}| |t|^{-3/2-2\delta}$$
for some $\delta>0$. Now if $\K\leq 1$, we use Lemma \ref{bornes sur zeta} to bound $\zeta(s+2)^{1-\K}$. Otherwise, we use the bound $(\zeta(s+2))^{-1} \ll \log(|t|)$ (see (II.3.56) of \cite{tenenbaum}). In both cases we get 
$$ Z_6^{(m)}(s) \ll_{m} |t|^{-3/2 -\delta}.$$
We now use Cauchy's formula for the derivatives, which states that 
$$ \zeta^{(k)}(s+1) = \frac{k!}{2\pi i} \oint_{|z|=r} \zeta(s+1+z) \frac{dz}{z^{k+1}}.$$
Selecting $r=\epsilon/2$ and applying Lemma \ref{bornes sur zeta}, we get the bound\footnote{This bound is still valid outside the zero-free region of $\zeta(s+1)$; this is why we considered the ordinary derivatives of $\zeta(s+1)$ instead of its logarithmic derivatives as with the other terms.}
$$ \zeta^{(k)}(s+1) \ll_{k,\epsilon} |t|^{1/2+\epsilon}.$$
We conclude the existence of $\delta>0$ such that  

$$Z^{(n)}(s) = \sum_{i=0}^n \binom{n}{i}\zeta^{(i)}(s+1) Z_6^{(n-i)}(s) \ll_n \frac{1}{|t|^{1+\delta}}. $$

\end{proof}

\subsubsection{The value of $\mu_{\K}(a,M)$}
\begin{proposition}
\label{k entier}
Assume Hypothesis \ref{hypothèse sur h}. If $\K\in \mathbb Z$, then
\begin{multline*}
\frac 1{2\pi i} \int_{(-1/2)}\frac{\mathfrak{S}_2(s) \zeta(s+1) \zeta(s+2)^{1-\K} Z_5(s)}{s(s+1)} M^s ds = 
 -\frac{\mu_{\K}(a,M)}M \left(1+O\left(\frac{\log\log M}{\log M}\right)\right) \\+ O_A\left(\frac 1{M\log^ A M} \right)
\end{multline*}
where $\mu_{\K}(a,M)$ is defined in Definition \ref{définition de mu_k}.
\end{proposition}
\begin{proof}

We first need to understand the behaviour of
\begin{align}
\label{représentation 1 de Z(s)}
Z(s) &:=   \frac{\mathfrak{S}_2(s) \zeta(s+1) \zeta(s+2)^{1-\K} Z_5(s)}{s(s+1)} \\
&= (s+1)^{\K+\omega_{\h}(a)-2} \frac{\mathfrak{S}_2(s)}{(s+1)^{\omega_{\h}(a)}}\zeta(s+1) ((s+1)\zeta(s+2))^{1-\K}  \frac{Z_5(s)}{s} 
\label{représentation 2 de Z(s)}
\end{align}
in the region $\mathcal D : -1 \leq  \Re s \leq -1/2$. This function is holomorphic for $\Re s > -1$ by Lemma \ref{transformée de Mellin}, and as we will see, the only point in $\mathcal D$ where $Z(s)$ is not necessarily locally bounded is $s=-1$. The functions
$$ \zeta(s+1), \hspace{1cm} ((s+1)\zeta(s+2))^{1-\K} \hspace{1cm} \text{and} \hspace{1cm} \frac 1s $$
are holomorphic on $\mathcal D$ and do not vanish at $s=-1$. The function $Z_5(s)$ is holomorphic for $\Re s >-1$, and all its derivatives are locally bounded around any point of $\mathcal D$ by Lemma \ref{bornes sur Z_5}. We compute
$$ Z_5(-1)= \prod_{p\nmid a} \frac{1-\h(p)/p}{(1-1/p)^{\K}}  \neq 0 ,$$
since $\h(p)<p$. As for the function $\mathfrak S_2(s)$, it is holomorphic on $\mathcal D$. However, this function can vanish at $s=-1$ if for a certain $p\mid a$ we have $\h(p^f)=\h(p^{f+1})/p$. In this case, we have for $s$ close to $-1$ that

\begin{multline*}
 \mathfrak S_2(s)  \prod_{\substack{p \mid a}}\left( 1-\frac 1{p^{s+2}}\right)^{\K-1}  = \prod_{\substack{p^f \parallel a: \\ \h(p^f)\neq \h(p^{f+1})/p, \\ f \geq 1} }\left[\frac{\h(p^f)-\h(p^{f+1})/p}{1-1/p}+O(|s+1|)\right]  \\
\times \prod_{\substack{p^f \parallel a: \\ \h(p^f)=\h(p^{f+1})/p, \\ f\geq 1} }[(s+1)(1+\h(p)+...+\h(p^{f}))\log p +O(|s+1|^2)],
\end{multline*}
and since $\h(p^e)\geq 0$, this shows that every local factor has at most a simple zero at $s=-1$. We conclude that  $$\frac{\mathfrak{S}_2(s)}{(s+1)^{\omega_{\h}(a)}}$$ is holomorphic on $\mathcal D$ and does not vanish at $s=-1$. We now split in three distinct cases, depending on the analytic nature of $(s+1)^{\K+\omega_{\h}(a)-2}$ near $s=-1$.

\textbf{First case: $\K+\omega_{\h}(a) \geq 2$.}
 In this case, $Z(s)$ and all of its derivatives are bounded near $s=-1$. To show this, note that it is true for the functions 
$$  (s+1)^{\K+\omega_{\h}(a)-2}, \hspace{0.5cm} \frac{\mathfrak{S}_2(s)}{(s+1)^{\omega_{\h}(a)}}, \hspace{0.5cm} \zeta(s+1), \hspace{0.5cm} ((s+1)\zeta(s+2))^{1-\K}, \hspace{0.5cm}\frac 1s \hspace{0.5cm} \text{and} \hspace{0.5cm} Z_5(s), $$
so it is also true for $Z(s)$ by Leibniz's rule.
We now shift the contour of integration to the left until the line $\Re s = -1+ \frac 1{\log M}$ to get
 \begin{align*}
 \frac 1{2\pi i} \int_{(-1/2)} Z(s) ds &= \frac i{2\pi i} \int_{\mathbb R} Z\left(-1+\frac 1{\log M} +it\right) M^{-1+\frac 1{\log M} + it} dt \\
&= \frac eM \frac 1{2\pi } \int_{\mathbb R} Z\left(-1+\frac 1{\log M} +it\right) e^{it \log M} dt,
 \end{align*}
which gives, after $A$ integrations by parts,
\begin{align*} \frac 1{2\pi i} \int_{(-1/2)} Z(s) ds  &\ll_A \frac 1 {M \log^A M} \int_{\mathbb R} \left|Z^{(A)}\left(-1+\frac 1{\log M} +it\right)\right| \left|e^{it \log M} \right| dt  \\
& \ll_{A} \frac 1 {M \log^A M}  \left( O(1) +\int_{|t|\geq 2} \frac 1 {|t|^{1+\delta}}  dt  \right)\\
& \ll_{A} \frac 1 {M \log^A M}
\end{align*}
by Lemma \ref{borne sur dérivées de Z}. Note that the uniformity in $\sigma$ was crucial. This shows that we can take $\mu_{\K}(a,M)=0$.

\textbf{Second case: $\K+ \omega_{\h}(a) = 1$.} Let $$ c:= \lim_{s\rightarrow 1^+} (s+1) Z(s) \neq 0$$
and define
$$ Z_7(s):=Z(s) -  \frac {c}{s+1}. $$
We can show using Lemma \ref{régularité de Z_5 en -1} that for $s$ close to $-1$ with $\Re s > -1$, the following bound holds:
$$Z_7(s) \ll 1. $$
Lemma \ref{régularité de Z_5 en -1} implies that for $s$ close to $-1$ with $\Re s > -1$, the function
$$ Z_7'(s) = \frac{((s+1)Z(s))'}{s+1} - \frac{(s+1)Z(s)}{(s+1)^2} +\frac {c}{(s+1)^2} $$
satisfies
$$ Z_7'(s) \ll \frac 1{|s+1|}.$$
 Using Lemma \ref{borne sur dérivées de Z}, we get that for $|t|\geq 2$, 
$$ Z_7'(s) \ll \frac 1 {|t|^{1+\delta}}. $$
Thus,
\begin{align*}
\frac 1{2\pi i}\int_{(-1+\frac 1 {\log M})}Z_7(s) M^s ds &= \frac {-1}{2\pi i \log M}\int_{(-1+\frac 1 {\log M})}Z'_7(s) M^s ds \\
& \ll \frac 1 {M\log M} \left|\int_{-\infty}^{\infty} Z'_7\left( -1+\frac 1 {\log M} + it \right) M^{it} dt\right| \\
& \ll \frac 1 {M\log M}\left( \left|\int_{-2}^{2} Z'_7\left( -1+\frac 1 {\log M} + it \right) M^{it} dt\right| + O(1) \right) \\ 
 & \ll  \frac 1 {M\log M}\left( \int_{-2}^{2} \frac 1{\frac 1 {\log M} + |t|} dt  + O(1) \right) \\
 & \ll \frac 1 {M\log M}\left( \int_{0}^{\frac 1{\log M}} \log M + \int_{\frac 1 {\log M}} ^2  \frac 1t dt  + O(1) \right) \\
&\ll \frac{\log \log M}{M\log M}.
\end{align*}
Combining this bound with an easy residue computation yields
\begin{align*}\frac 1{2\pi i} \int_{(-1/2)} Z(s)  M^s ds &= \frac 1{2\pi i} \int_{(-1/2)} Z_7(s)  M^s ds + \frac 1{2\pi i} \int_{(-1/2)}\frac { c}{s+1} M^s ds\\
 &= \frac{c}M \left( 1+O\left( \frac {\log \log M} { \log M} \right)\right) . 
\end{align*}
Now remarks \ref{remarque k=0} and \ref{remarque k=1} show that $c=-\mu_{\K}(a,M)$, which concludes this case.

\textbf{Third case: $\K=\omega_{\h}(a)=0$.}
Defining
$$c:= \lim_{s \rightarrow -1^+} (s+1)^2 Z(s) \neq 0, $$
we get that the function $Z_8(s):=Z(s)-\frac{c}{(s+1)^2}$ satisfies the bound $$Z_8(s) \ll \frac 1{|s+1|}$$ by Lemma \ref{régularité de Z_5 en -1}. 
An easy residue computation yields
\begin{equation*}
\frac 1{2\pi i} \int_{(-1/2)}Z(s) M^s ds=
c \frac{\log M}{M} + \frac 1{2\pi i} \int_{(-1+\frac 1 {\log M})}Z_8(s) M^s ds.
\end{equation*}
Proceeding in an analogous way to the previous case, we compute
\begin{align*}
\int_{(-1+\frac 1 {\log M})}Z_8(s) M^s ds & \ll \frac 1 M \left|\int_{-\infty}^{\infty} Z_8\left( -1+\frac 1 {\log M} + it \right) M^{it} dt\right| \\
& \ll \frac 1 M\left( \left|\int_{-2}^{2} Z_8\left( -1+\frac 1 {\log M} + it \right) M^{it} dt\right| + O(1) \right) \\ 
&\ll \frac{\log \log M}{M},
\end{align*}
from which we conclude
\begin{align*}
\frac 1{2\pi i} \int_{(-1/2)}Z(s) M^s ds&=c  \frac{\log M}{M}\left(1+ \frac{\log \log M}{ \log M} \right)\\
& =-\frac{\mu_0(a,M)}M \left(1+ \frac{\log \log M}{ \log M} \right)
\end{align*}
by Remark \ref{remarque k=0}, since
$$c= \frac 1{2}\prod_{\substack{p^f\parallel a \\ f\geq 0}} (\h(p^f)-\h(p^{f+1})/p).$$

\end{proof}

\begin{lemma}
 \label{hankel}
Let $z> 1$ be a real number. Then,
$$ \frac 1 {2\pi i} \int_{\Re s = -1/2} \frac{M^s}{(s+1)^z} ds = \frac 1M\frac{(\log M)^{z-1}}{\Gamma(z)}.$$
\end{lemma}
\begin{proof}

Let $R\geq 2$ be a large real number and consider $\mathcal H_R$ a Hankel contour centered at $s=-1$ and truncated at $-R\pm \epsilon i$. Define $C_R$ to be the union of two circle segments starting at the endpoints of $\mathcal H_R$ and ending at the points $ \pm iR$. By Cauchy's formula,
\begin{align*}
  \frac 1 {2\pi i} \int_{\Re s = -1/2} \frac{M^s}{(s+1)^z} ds &=  \frac 1 {2\pi i} \int_{\Re s = 0} \frac{M^s}{(s+1)^z} ds  \\
&= \frac 1 {2\pi i} \int_{\mathcal H_R} \frac{M^s}{(s+1)^z} ds +  \frac 1 {2\pi i} \int_{C_R} \frac{M^s}{(s+1)^z} ds   \\
& = \frac 1 {2\pi i} \int_{\mathcal H_R} \frac{M^s}{(s+1)^z} ds + O\left( \frac 1 {R^{z-1}} \right),
\end{align*}
so by taking $R \rightarrow \infty$, 

\begin{align*}
  \frac 1 {2\pi i} \int_{\Re s = -1/2} \frac{M^s}{(s+1)^z} ds &= \frac 1 {2\pi i} \int_{\mathcal H_{\infty}} \frac{M^s}{(s+1)^z} ds  \\
& = \frac 1 M \frac 1 {2\pi i} \int_{\mathcal H_{\infty}} \frac{e^{(s+1)\log M}}{(s+1)^z} ds \\
&=  \frac {(\log M)^{z-1}} M \frac 1 {2\pi i} \int_{\mathcal H'_{\infty}} \frac{e^{w}}{w^z} dw \\
&= \frac 1M \frac {(\log M)^{z-1}}{\Gamma(z)}
\end{align*}
by Hankel's formula (see Théorème II.0.17 of \cite{tenenbaum}). Here, $\mathcal H'_{\infty}$ denotes an infinite Hankel contour centered at $w=0$.
\end{proof}

\begin{proposition}
\label{k non entier}
Assume Hypothesis \ref{hypothèse sur h}. If $\K\notin \mathbb Z$, then 
\begin{equation*}
\frac 1{2\pi i} \int_{(-1/2)}\frac{ \mathfrak{S}_2(s) \zeta(s+1) \zeta(s+2)^{1-\K} Z_5(s)}{s(s+1)} M^s ds = -\frac{\mu_{\K}(a,M)}M\left(1+ O\left(\frac 1{\log M}\right)\right).
\end{equation*}
\end{proposition}
\begin{proof}
As in Proposition \ref{k entier}, we need to study the function
$$Z(s)= (s+1)^{\K+\omega_{\h}(a)-2} \frac{\mathfrak{S}_2(s)}{(s+1)^{\omega_{\h}(a)}} \zeta(s+1) ((s+1)\zeta(s+2))^{1-\K}  \frac{Z_5(s)}{s}$$
in the region $\mathcal D : -1 \leq  \Re s \leq -1/2$. This function is holomorphic for $\Re s > -1$ by Lemma \ref{bornes sur Z_5}, and the only point in  $\mathcal D$ where $Z(s)$ is not necessarily locally bounded is $s=-1$. However, the functions
$$ \frac{\mathfrak{S}_2(s)}{(s+1)^{\omega_{\h}(a)}}, \hspace{1cm} \zeta(s+1), \hspace{1cm}  ((s+1)\zeta(s+2))^{1-\K} \hspace{1cm} \text{and} \hspace{1cm} \frac 1s $$
are holomorphic on $\mathcal D$ and do not vanish at $s=-1$. The function $Z_5(s)$ is holomorphic for $\Re s >-1$, all its derivatives are locally bounded around any point of $\mathcal D$, and $Z_5(-1) \neq 0$.
Define
 $$ Z_9(s):=Z(s) - c (s+1)^{\K+\omega_{\h}(a)-2}, $$
where 
$$c:=\lim_{s\rightarrow -1^+} (s+1)^{2-\K-\omega_{\h}(a)}Z(s)\neq 0. $$
We have that
\begin{align}
  \label{hankel equation}
\begin{split}
\frac 1{2\pi i} \int_{(-1/2)}Z(s) M^s ds &=\frac {(-1)^{\lceil \K \rceil + \omega_{\h}(a)}}{2\pi i (\log M)^{\lceil \K \rceil + \omega_{\h}(a)}} \int_{(-1/2)}Z^{(\lceil \K \rceil + \omega_{\h}(a))}(s) M^s ds \\
&= \frac {(-1)^{\lceil \K \rceil + \omega_{\h}(a)}}{2\pi i (\log M)^{\lceil \K \rceil + \omega_{\h}(a)}} \left( \int_{(-1/2)}Z^{(\lceil \K \rceil + \omega_{\h}(a))}_{9}(s) M^s ds  \right. \\
 & \hspace{1cm}\left.+ c\frac{\Gamma(\K+\omega_{\h}(a)-1 )}{ \Gamma( \K-\lceil \K \rceil-1 )} \int_{(-1/2)} (s+1)^{\K-\lceil \K \rceil-2}  M^s ds\right)  \\
& =  \frac cM \frac{(\log M)^{1-\K-\omega_{\h}(a)}}{\Gamma(2-\K-\omega_{\h}(a))}  \\
& \hspace{1cm}  + \frac {(-1)^{\lceil \K \rceil + \omega_{\h}(a)}}{2\pi i (\log M)^{\lceil \K \rceil + \omega_{\h}(a)}}\int_{(-1/2)}Z^{(\lceil \K \rceil + \omega_{\h}(a))}_{9}(s) M^s ds
\end{split}
\end{align}
by Lemma \ref{hankel}.
We will show the bound
\begin{equation} \label{borne sur dérivées de Z_9} Z_9^{(\lceil \K \rceil + \omega_{\h}(a))}(s) \ll |s+1|^{\K-\lceil \K \rceil-1}
 \end{equation}
for $s$ close to $-1$, which will yield (using Lemma \ref{borne sur dérivées de Z})
\begin{align*}
\int_{(-1+\frac 1 {\log M})}Z_9^{(\lceil \K \rceil + \omega_{\h}(a))}(s) M^s ds  & \ll \frac 1 M \left|\int_{-\infty}^{\infty} Z_9^{(\lceil \K \rceil + \omega_{\h}(a))}\left( -1+\frac 1 {\log M} + it \right) M^{it} dt\right| \\
&= \frac 1 M \left|\int_{-2}^{2} Z_9^{(\lceil \K \rceil + \omega_{\h}(a))}\left( -1+\frac 1 {\log M} + it \right) M^{it} dt+O(1)\right|  
  \\ & \ll  \frac 1 M\left( \int_{-2}^{2} \left(\frac 1 {\log M} + |t|\right)^{\K-\lceil \K \rceil-1} dt  + O(1) \right) \\
 & \ll \frac 1 M\left( \int_{0}^{\frac 1{\log M}} (\log M)^{1-\K+\lceil \K \rceil} + \int_{\frac 1 {\log M}} ^2  t^{\K-\lceil \K \rceil-1} dt  + O(1) \right) \\
&\ll \frac{(\log M)^{\lceil \K\rceil- \K }+1}{M} \ll \frac{(\log M)^{\lceil \K\rceil- \K }}{M},
\end{align*}
from which we will conclude using \eqref{hankel equation} that
\begin{align*}
\frac 1{2\pi i} \int_{(-1/2)}Z(s) M^s ds &=
\frac cM\frac{(\log M)^{1-\K-\omega_{\h}(a)}}{\Gamma(2-\K-\omega_{\h}(a))} \left(1+ O\left(\frac 1 {\log M} \right)\right) \\
&= -\mu_{\K}(a,M) \left(1+ O\left(\frac 1 {\log M} \right)\right),
\end{align*}
achieving the proof. Let us now show that \eqref{borne sur dérivées de Z_9} holds. By Lemma \ref{bornes sur Z_5}, the function
$$ Z_{10}(s):= (s+1)^{2-\K-\omega_{\h}(a)} Z(s)$$
as well as its derivatives are locally bounded around $s=-1$. Moreover, applying Lemma \ref{régularité de Z_5 en -1} gives the bound
\begin{equation}
 \label{borne sur Z_10} 
Z_{10}(s) = Z_{10}(-1)+O( |s+1|).
\end{equation}
Now we use Leibniz's formula: 
\begin{align*}
 Z^{(\lceil \K \rceil + \omega_{\h}(a))}(s) &= \left( (s+1)^{\K+\omega_{\h}(a)-2} Z_{10}(s)  \right)^{(\lceil \K \rceil + \omega_{\h}(a))} \\
& = \sum_{i=0}^{\lceil \K \rceil + \omega_{\h}(a)} \binom{\lceil \K \rceil + \omega_{\h}(a)}{i} \left((s+1)^{\K+\omega_{\h}(a)-2}\right)^{(i)} Z_{10}^{(\lceil \K \rceil + \omega_{\h}(a)-i)}(s) \\
& = \left((s+1)^{\K+\omega_{\h}(a)-2}\right)^{(\lceil \K \rceil + \omega_{\h}(a))}Z_{10}(s) + O(|s+1|^{\K-\lceil \K \rceil-1}) \\
& = \left((s+1)^{\K+\omega_{\h}(a)-2}\right)^{(\lceil \K \rceil + \omega_{\h}(a))}Z_{10}(-1) + O(|s+1|^{\K-\lceil \K \rceil-1})
\end{align*}
by \eqref{borne sur Z_10}, so
\begin{align*}
 Z_9^{(\lceil \K \rceil + \omega_{\h}(a))}(s) &= Z^{(\lceil \K \rceil + \omega_{\h}(a))}(s) - c \left((s+1)^{\K+\omega_{\h}(a)-2}\right)^{(\lceil \K \rceil + \omega_{\h}(a))} \\
& = (Z_{10}(-1)- c) \left((s+1)^{\K+\omega_{\h}(a)-2}\right)^{(\lceil \K \rceil + \omega_{\h}(a))} + O(|s+1|^{\K-\lceil \K \rceil-1}) \\
& = O(|s+1|^{\K-\lceil \K \rceil-1}) 
\end{align*}
since $c=Z_{10}(-1)$.

\end{proof}

\begin{lemma}
 \label{partie négligeable}
Assume Hypothesis \ref{hypothèse sur h}. Let $y\geq 1$ be a real number. Then, 
\begin{equation} \frac 1{2\pi i} \int_{(-1/2)}\mathfrak{S}_2(s) \zeta(s+1) \zeta(s+2)^{1-\K} Z_5(s) y^s \frac{ds}{s}  \ll_{\epsilon} y^{-1+\epsilon}.
\label{borne a prouver partie negligeable}\end{equation}
\end{lemma}

\begin{proof}
Define
$$ Z_{\A}(s):= \mathfrak{S}_2(s) \zeta(s+1) \zeta(s+2)^{1-\K} Z_5(s).$$
The goal is to bound the integral
$$ \frac 1{2\pi i} \int_{(-1/2)} Z_{\A}(s) y^s \frac{ds}{s} = \frac 1{2\pi i} \int_{(-1+\epsilon)} Z_{\A}(s) y^s \frac{ds}{s}. $$
We will first show that this integral is $\ll_{\epsilon} y^{-1/2+\epsilon}$ using complex analysis, and then we will see how to improve this bound to $\ll_{\epsilon}y^{-1+\epsilon} $ by elementary means.
In the region $-1+\epsilon< \sigma $, we have the bound
$$ |Z(\sigma+it)| \ll_{\epsilon} |\zeta(\sigma+1+it)| \ll_{\epsilon} (|t|+2)^{\mu(\sigma+1)+\epsilon},$$
where $ \mu(\sigma+1)$ is defined as in Lemma \ref{bornes sur zeta}. Thus we get the bounds
$$ \int_{-1+\epsilon-iT}^{-1+\epsilon+iT} Z(s) y^s \frac{ds}{s}  \ll_{\epsilon} \frac{T^{1/2}}{y^{1-\epsilon}},$$
$$ \int_{-1+\epsilon \pm iT}^{\epsilon \pm iT} Z(s)y^s \frac{ds}{s} \ll_{\epsilon} (Ty)^{\epsilon} \left( \frac 1{T^{5/6} y^{1/2}} + \frac 1{T^{1/2}y} + \frac 1T + \frac 1{T^{5/6} y^{1/2}} \right).  $$
The last integral we need to bound is
\begin{align*}
 \frac 1 {2\pi i} \int_{\Re s =  \epsilon, |\Im s|>T } Z(s) y^s \frac{ds}{s} &= \sum_n \frac{\f_a(n)}{n\gamma(n)} \frac 1{2\pi i} \int_{\Re s = \epsilon, |\Im s|>T } \left(\frac yn\right)^s \frac{ds}{s} \\
&\ll y^{\epsilon} \sum_n \frac{\f_a(n)}{n\gamma(n)} \frac 1{n^{\epsilon} (1+T|\log(y/n)|)}
\end{align*}
by the effective version of Perron's formula (see Théorème II.2.3 of \cite{tenenbaum}). The last sum is
\begin{align*}
& \ll \frac {y^{\epsilon}}{\sqrt T} \sum_{n \leq y\left(1-\frac 1{\sqrt T}\right)} \frac{\f_a(n)}{n\gamma(n)}  + \sum_{  y\left(1-\frac 1{\sqrt T}\right)\leq n \leq y\left(1+\frac 1{\sqrt T}\right)} \frac{\f_a(n)}{n\gamma(n)} +  \frac {y^{\epsilon}}{\sqrt T} \sum_{n \geq y\left(1+\frac 1{\sqrt T}\right)} \frac{\f_a(n)}{n\gamma(n)} \frac 1{n^{\epsilon}}  \\
& \ll \frac {y^{\epsilon}}{\sqrt T} \log y + \frac 1  {\sqrt T} \ll_{\epsilon} \frac {y^{\epsilon}}{\sqrt T} \log y
\end{align*}
by Lemma \ref{borne sur g}. Taking $T=y$ yields that the left hand side of \eqref{borne a prouver partie negligeable} is $\ll_{\epsilon} y^{-1/2+\epsilon}$. We now proceed to show this bound can be improved to $\ll_{\epsilon} y^{-1+\epsilon}$.  
The function $Z_{\A}(s) y^s/s$ has a double pole at $s=0$ with residue equal to $C_1\log y+C_2$, where $C_1$ and $C_2$ are real numbers independent of $y$. By the residue theorem and Mellin inversion,
\begin{align}
\label{retour a la somme arith.}
\begin{split}
 \frac 1{2\pi i} \int_{(-1/2)} Z_{\A}(s) y^s \frac{ds}{s}  &=- C_1\log y-C_2+ \frac 1{2\pi i} \int_{(1)} Z_{\A}(s) y^s \frac{ds}{s} \\
 &=  \sum_{n \leq y} \frac{\f_a(n)}{n \gamma(n)} -C_1\log y-C_2.
\end{split}
\end{align}

Let us give an elementary estimate for the sum appearing on the right hand side of \eqref{retour a la somme arith.}. Define
$$ \nu(n) := \prod_{p\mid n }\frac{1-\h(p)}{p-1}.  $$ 
Using the convolution identity $$\frac 1{\gamma(n)} = \sum_{rs=n} \mu^2(s)\nu(s),$$
we compute
\begin{align}
\begin{split}
 \sum_{n \leq y} \frac{\f_a(n)}{n \gamma(n)}  &= \sum_{s\leq y} \frac{\mu^2(s)\nu(s)}{s} \sum_{r\leq y/s} \frac{\f_a(rs)}{r} =\sum_{s\leq y}\frac{\mu^2(s)\nu(s)}{s}\sum_{(a,s)\mid d \mid a} \f_a(d) \sum_{\substack{r\leq y/s: \\ (a,rs)=d}} \frac 1r \\
&= \sum_{s\leq y}\frac{\mu^2(s)\nu(s)}{s} \sum_{\substack{(a,s)\mid d \mid a  }}\f_a(d) \sum_{\substack{r\leq y/s : \\ \frac{d}{(d,s)}\mid r \\ (a,rs)=d}} \frac 1r \\
&= \sum_{s\leq y}\frac{\mu^2(s)\nu(s)}{s} \sum_{\substack{(a,s)\mid d \mid a  }}\f_a(d) \frac{(d,s)}{d} \sum_{\substack{l\leq \frac{y(d,s)}{ds} : \\  (a/d,ls/(d,s))=1}} \frac 1l \\
&= \sum_{s\leq y}\frac{\mu^2(s)\nu(s)}{s} \sum_{\substack{(a,s)\mid d \mid a: \\ (a/d,s/(d,s))=1 }}\f_a(d) \frac{(d,s)}{d} \sum_{\substack{l\leq \frac{y (d,s)}{ds}:\\ (l,a/d)=1}} \frac 1l \\
&=  \sum_{s\leq y}\frac{\mu^2(s)\nu(s)}{s} \sum_{\substack{(a,s)\mid d \mid a : \\ (a/d,s/(d,s))=1 }}\f_a(d) \frac{(d,s)}{d} \frac{\phi(a/d)}{a/d} \bigg( \log\left(\frac{y (d,s)}{ds} \right)+\gamma  \\
& \hspace{3cm} +\sum_{p\mid a/d} \frac{\log p}{p-1}  + O\left( \frac{ds}{y(d,s)}\right)\bigg).
\end{split}
\label{grosse equation dans calcul elementaire}
\end{align}
Using the bound $ \nu(n) \ll_{\epsilon} n^{-1+\epsilon}, $ which is deduced from Hypothesis \ref{hypothèse sur h}, we get that the error terms sum to $O_{a,\epsilon}(y^{-1+\epsilon})$. Moreover, we can extend the sum over $s\leq y$ to all integers, at the cost of the error term $O_{a,\epsilon}(y^{-1+\epsilon})$. Having done this, \eqref{grosse equation dans calcul elementaire} becomes
\begin{equation}\sum_{n \leq y} \frac{\f_a(n)}{n \gamma(n)} = \tilde{C}_1 \log y+\tilde{C}_2  + O_{a,\epsilon}(y^{-1+\epsilon}),\label{bonne borne dans partie negligeable}\end{equation}
where $\tilde{C}_1$ and $\tilde{C}_2$ are real numbers which do not depend on $y$. Substituting \eqref{bonne borne dans partie negligeable} into \eqref{retour a la somme arith.} and using our previous bound, we get 
$$ (\tilde{C}_1-C_1) \log y+\tilde{C}_2-C_2  + O_{a,\epsilon}(y^{-1+\epsilon}) = \frac 1{2\pi i} \int_{(-1/2)} Z_{\A}(s) y^s \frac{ds}{s}  = O_{a,\epsilon}(y^{-1/2+\epsilon}), $$
which of course implies that $\tilde{C}_1=C_1$ and $\tilde{C}_2=C_2$ since these numbers do not depend on $y$. We conclude from \eqref{retour a la somme arith.} and \eqref{bonne borne dans partie negligeable} that \eqref{borne a prouver partie negligeable} holds.

\end{proof}

\subsubsection{Proof of Proposition \ref{proposition evaluation de la somme arithmetique}}

\begin{proof}[Proof of Proposition \ref{proposition evaluation de la somme arithmetique}]
First we use Lemma \ref{transformée de Mellin} to write

\begin{align*}
 S_5:&=\sum_{1\leq r \leq  \R}\frac{\f_a(r)}{r\gamma(r)}\left( 1-\frac{r}{\R} \right)-\sum_{1\leq r \leq  M}\frac{\f_a(r)}{r\gamma(r)}\left(1- \frac{r}{M} \right) - \sum_{\frac xR< q \leq  \frac x M}\frac{\f_a(q)}{q\gamma(q)} \\
&=\frac 1{2\pi i} \int_{(1)}\mathfrak{S}_2(s) \zeta(s+1) \zeta(s+2)^{1-\K} Z_5(s) \left(\frac {{\R}^s-M^s}{s+1}  +\left(\frac{x}{\R}\right)^s-\left(\frac{x}{M}\right)^s   \right) \frac{ds}s .
\end{align*}

Writing $$\psi(s):=\frac {\R^s-M^s}{s+1}  +\left(\frac{x}{\R}\right)^s-\left(\frac{x}{M}\right)^s, $$ it is trivial that $\psi(0)=0$. Using Taylor series, we have for $s$ close to $0$ that
$$ \psi(s) = (1+O(s)) (s\log(\R/M)+O(s^2)) + s\log(x/\R)-s\log(x/M)+O(s^2),$$
which means that $\psi$ has a double zero at $s=0$. Thus, $$\mathfrak{S}_2(s)\zeta(s+1) \zeta(s+2)^{1-\K} Z_5(s) \frac{\psi(s)}s$$ is holomorphic at $s=0$. Using this fact,
\begin{align*}
S_5 &=  \frac 1{2\pi i} \int_{(-1/2)}\mathfrak{S}_2(s) \zeta(s+1) \zeta(s+2)^{1-\K} Z_5(s) \psi(s) \frac{ds}s
\\ &= \frac 1{2\pi i} \int_{(-1/2)}\mathfrak{S}_2(s) \zeta(s+1) \zeta(s+2)^{1-\K} Z_5(s) (\R^s-M^s) \frac{ds}{s(s+1)} + O_{\epsilon}\left( \left(\frac{\R}{x} \right)^{1-\epsilon} \right)
\end{align*}
by Lemma \ref{partie négligeable}. We conclude using propositions \ref{k entier} and \ref{k non entier} that
\begin{align*}
 S_5 &= \frac{\mu_{\K}(a,M)}{M} \left( 1+O\left( \frac{\log\log M}{\log M}\right) \right) + O_A\left( \frac 1 {M\log^ A M}\right) \\
&\hspace{1cm}-\frac{\mu_{\K}(a,\R)}{\R} \left( 1+O\left( \frac{\log\log \R}{\log \R}\right) \right)+ O_A\left( \frac 1 {\R\log^ A \R}\right)+ O_{\epsilon}\left( \left(\frac{\R}{x} \right)^{1-\epsilon} \right) \\
&= \frac{\mu_{\K}(a,M)}{M} \left( 1+O\left( \frac{\log\log M}{\log M}\right) \right) + O_A\left( \frac 1 {M\log^ A M}\right)
\end{align*}
since $M(x)^{1+\delta} \leq  \LL(x)^{1+\delta} \leq \R(x) \leq \sqrt x $.
\end{proof}

\subsection{Proofs of theorems \ref{theoreme principal} and \ref{theoreme principal}*}

We first define the following counting function, which will come in handy for the proofs of this section:
\begin{equation}\A^*(x;q,a):=\sum_{\substack{|a| < n\leq x \\ n\equiv a \bmod q}} \a(n).
\label{def de A*}
\end{equation}

\begin{proof}[Proof of Theorem \ref{theoreme principal}]
 
Let $1\leq M(x) \leq \LL(x)$ and let $\R=\R(x)$ be as in Hypothesis \ref{bombieri-vinogradov}. We decompose the sum \eqref{equation thm principal} as follows:
\begin{align}
\label{séparation en quatre sommes}
\begin{split}
 \sum_{q\leq \frac xM} &\left( \A(x;q,a)-\a(a)-\frac{\f_a(q)}{q\gamma(q)} \A(x) \right)  = \sum_{q\leq \frac xM} \left( \A^*(x;q,a) -\frac{\f_a(q)}{q\gamma(q)} \A(x) \right)+O(1)\\
  &=\sum_{\frac x{\R} <q \leq x}\A^*(x;q,a)-\sum_{\frac xM <q \leq x}\A^*(x;q,a)  -\A(x)\sum_{ \frac x {\R} < q\leq \frac xM}\frac{\f_a(q)}{q\gamma(q)}  \\ 
& \hspace{4cm}+  \sum_{q\leq \frac x{\R}}\left( \A^*(x;q,a)-\frac{\f_a(q)}{q\gamma(q)}\A(x) \right) +O(1) \\
 &=S_1-S_2-S_3 + S_4+O(1).
\end{split} 
\end{align}
Hypothesis \ref{bombieri friedlander iwaniec} implies the bound $$S_4 \ll \frac{\A(x)}{M(x)^{1+\delta}}.  $$ 
To evaluate the sums $S_1$ and $S_2$ we use the Hooley-Montgomery divisor switching technique (see \cite{hooley}). Setting $n=a+qr$, we have for positive $a$ that
\begin{multline}
 S_2= \sum_{\frac xM <q \leq x}\sum_{\substack{|a|<n\leq x \\ n\equiv a \bmod q}} \a(n) = \sum_{1\leq r<(x-a)\frac Mx} \sum_{\substack{a+ r\frac xM<n\leq x \\ n\equiv a \bmod r}} \a(n) \\= \sum_{1\leq r<(x-a)\frac Mx} \left( \A(x;r,a)-\A\left(a+ r \frac xM;r,a  \right)  \right).
\label{chose a changer si a est negatif}
\end{multline}
Using Hypothesis \ref{bombieri-vinogradov}, we see that there exists $\delta>0$ such that
\begin{align}
\begin{split}
 S_2&=  \sum_{1\leq r<(x-a)\frac Mx}\frac{\f_a(r)}{r\gamma(r)} \left(\A(x)-\A\left(a+r \frac xM\right)  \right)+ O\left(\frac{\A(x)}{\LL(x)^{1+2\delta}} \right) \\
&= \sum_{1\leq r<(x-a)\frac Mx}\frac{\f_a(r)}{r\gamma(r)} \left(\A(x)-\A\left( \frac rM x\right)  \right)+ O\left(\frac{\A(x)}{\LL(x)^{1+\delta}} \right)  \\
&= \A(x) \sum_{1\leq r<(x-a)\frac Mx}\frac{\f_a(r)}{r\gamma(r)} \left(1-\frac{\A\left( \frac rM x\right)}{\A(x)}  \right)+ O\left(\frac{\A(x)}{\LL(x)^{1+\delta}} \right)
\end{split} 
\label{chose a ne pas changer si a est negatif}
\end{align}
by hypotheses \ref{taille de A} and Lemma \ref{borne sur g}. Now, if $a$ were negative, we would have to add an error term of size $\ll \frac{\A(x)}{\LL(x)^{1+\delta}}$ to \eqref{chose a changer si a est negatif} (by Hypothesis \ref{taille de A}), which would yield the same error term in \eqref{chose a ne pas changer si a est negatif}. Using Hypothesis \ref{taille de A} again, \eqref{chose a ne pas changer si a est negatif} becomes
\begin{equation*}
 = \A(x) \sum_{1\leq r<(x-a)\frac Mx}\frac{\f_a(r)}{r\gamma(r)} \left(1-\frac rM  \right)+ O\left(\frac{\A(x)}{\LL(x)^{1+\delta}} \right) .
\end{equation*}
If $M$ is an integer, then the $M$-th term of the sum is $\frac{\f_a(r)}{r\gamma(r)}\left( 1 -\frac MM\right) =0$. If not, the bound $\frac{\f_a(r)}{r\gamma(r)} \ll_{\epsilon} \frac 1{\phi(r)}$ (see Lemma \ref{borne sur g}) implies that this last term is $\ll \A(x) \frac{\log \log M} {M^2}$. Thus, 
$$S_2= \A(x) \sum_{1\leq r \leq  M}\frac{\f_a(r)}{r\gamma(r)} \left( 1- \frac{r}{M} \right)+ O\left(\frac{\A(x)}{M^{1+\delta}}\right) $$
since $M(x) \leq \LL(x).$ A similar calculation shows that $$S_1= \A(x) \sum_{1\leq r \leq  \R(x)}\frac{\f_a(r)}{r\gamma(r)} \left( 1-\frac{r}{\R(x)} \right)+ O\left(\frac{\A(x)}{\LL(x)^{1+\delta}}\right). $$
Grouping terms, \eqref{séparation en quatre sommes} becomes
\begin{align*}
 \sum_{q\leq \frac xM} &\left( \A(x;q,a)-\a(a)-\frac{\f_a(q)}{q\gamma(q)} \A(x) \right)=S_1-S_2-S_3+S_4 +O(1)  \\
&=\A(x) \left(\sum_{1\leq r \leq  \R}\frac{\f_a(r)}{r\gamma(r)}\left( 1-\frac{r}{\R} \right)-\sum_{1\leq r \leq  M}\frac{\f_a(r)}{r\gamma(r)}\left(1- \frac{r}{M} \right) - \sum_{\frac x{\R}< q \leq  \frac x M}\frac{\f_a(q)}{q\gamma(q)}  \right)  \\
& \hspace{3cm}+O\left( \frac {\A(x)}{M^{1+\delta}}\right),
\end{align*}
which combined with Proposition \ref{proposition evaluation de la somme arithmetique} gives 
$$ = \frac{\A(x)}{M} \mu_{\K} (a,M) \left( 1+O\left( \frac{\log\log M}{\log M}\right) \right) +O_A\left( \frac{\A(x)}{M \log^A M} \right), $$
that is
\begin{multline*}
 \sum_{q\leq \frac xM} \left( \A(x;q,a)-\a(a)-\frac{\f_a(q)}{q\gamma(q)} \A(x) \right)  \\
 =  \frac {\A(x)}M \left( \mu_{\K}(a,M)\left( 1+O\left( \frac{\log\log M}{\log M}\right) \right) +O_A \left( \frac 1 {\log^ A M}\right)\right) .
\end{multline*}
\end{proof}

\begin{proof}[Proof of Theorem \ref{theoreme principal}*]
 
Let $1\leq M(x) \leq \LL(x)$ and let $\R=\R(x)$ be as in Hypothesis \ref{bombieri friedlander iwaniec}. We decompose the sum \eqref{equation thm principal dyadique} as follows:
\begin{align}
\begin{split}
 \sum_{ \frac x{2M}< q\leq \frac xM} &\left( \A(x;q,a)-\a(a)-\frac{\f_a(q)}{q\gamma(q)} \A(x) \right)  = \sum_{ \frac x{2M}< q\leq \frac xM} \left( \A^*(x;q,a) -\frac{\f_a(q)}{q\gamma(q)} \A(x) \right)+O(1)\\
  &=\sum_{\frac x{2M} <q \leq x}\A^*(x;q,a)-\sum_{\frac xM <q \leq x}\A^*(x;q,a) -\A(x)\sum_{ \frac x {2M} < q\leq \frac xM}\frac{\f_a(q)}{q\gamma(q)} +O(1) \\
 &=S_1-S_2-S_3 +O(1).
\end{split} 
\label{séparation en quatre sommes dyadique}
\end{align}
Arguing as in the proof of Theorem \ref{theoreme principal}, we set $n=a+qr$ to get that for positive $a$,
\begin{align*}
 S_2 &= \sum_{1\leq r<(x-a)\frac Mx} \left( \A(x;r,a)-\A\left(a+ r \frac xM;r,a  \right)  \right) \\
&= \A(x) \sum_{1\leq r<(x-a)\frac Mx}\frac{\f_a(r)}{r\gamma(r)} \left(1-\frac{\A\left( \frac rM x\right)}{\A(x)}  \right)+ O\left(\frac{\A(x)}{\LL(x)^{1+\delta}} \right)  \\
&=\A(x) \sum_{1\leq r \leq  M}\frac{\f_a(r)}{r\gamma(r)} \left( 1- \frac{r}{M} \right)+ O\left(\frac{\A(x)}{M^{1+\delta}}\right) 
\end{align*}
by Hypotheses \ref{bombieri-vinogradov}*, \ref{taille de A} and Lemma \ref{borne sur g}. Now, if $a$ were negative, we would have to add a negligible contribution. Thus, \eqref{séparation en quatre sommes dyadique} becomes
\begin{align*}
 \sum_{\frac x{2M}< q\leq \frac xM} &\left( \A(x;q,a)-\a(a)-\frac{\f_a(q)}{q\gamma(q)} \A(x) \right)=\A(x) \bigg(\sum_{1\leq r \leq  2M}\frac{\f_a(r)}{r\gamma(r)}\left( 1-\frac{r}{2M} \right) \\
&-\sum_{1\leq r \leq  M}\frac{\f_a(r)}{r\gamma(r)}\left(1- \frac{r}{M} \right) - \sum_{\frac x{2M}< q \leq  \frac x M}\frac{\f_a(q)}{q\gamma(q)}  \bigg)+O\left( \frac {\A(x)}{M^{1+\delta}}\right).
\end{align*}
Going through the proof of Proposition \ref{proposition evaluation de la somme arithmetique}, we see that this is 
\begin{multline*} = \frac{\A(x)}{M} \mu_{\K} (a,M) \left( 1+O\left( \frac{\log\log M}{\log M}\right) \right)- \frac{\A(x)}{2M} \mu_{\K} (a,2M) \left( 1+O\left( \frac{\log\log M}{\log M}\right) \right) \\
 +O_A\left( \frac{\A(x)}{M \log^A M} \right),
\end{multline*}

that is
\begin{multline*}
 \sum_{\frac x{2M}< q\leq \frac xM} \left( \A(x;q,a)-\a(a)-\frac{\f_a(q)}{q\gamma(q)} \A(x) \right)   \\
=  \frac {\A(x)}{2M} \left( \mu_{\K}(a,M)\left( 1+O\left( \frac{\log\log M}{\log M}\right) \right) +O_A \left( \frac 1 {\log^ A M}\right)\right),
\end{multline*}
since by the definition of $\mu_{\K}(a,M)$, 
\begin{equation*} 2\mu_{\K}(a,M)-\mu_{\K}(a,2M) = \mu_{\K}(a,M) \left( 1+O\left( \frac 1{\log M}\right)\right).
\end{equation*}

\end{proof}

\section{Further Proofs}
\label{section preuve des exemples}
In this section we prove the results of Section \ref{section exemples}.

\begin{proof}[Proof of Theorem \ref{resultat premiers}]

Put
$$ a(n):= \Lambda(n), $$
which gives $\A(x)=\psi(x)$ and $\A(x;q,a)=\psi(x;q,a)$. Define 
$$\f_a(q):=\begin{cases}
		1 &\text{ if } (a,q)=1 \\                                                                            
		0 &\text{ otherwise, }        
\end{cases}$$
and $\gamma(q):= \frac{\phi(q)} q$. Define also the multiplicative function $\h(d)$ by $\h(1)=1$, and $\h(d)=0$ for $d>1$. The prime number theorem in arithmetic progressions gives the asymptotic
$$ \A(x;q,a) \sim   \frac{\f_a(q)}{q \gamma(q)} \A(x), $$
for any fixed $a$ and $q$ such that $(a,q)=1$. Now let us show that the hypotheses of Section \ref{hypothèses} hold. Fix $A>0$ and put $\LL(x):= (\log x)^A$, $\R(x):=x^{1/2}(\log x)^{-B(A)}$, where $B(A):=A+5$. Hypothesis \ref{bombieri-vinogradov} is the Bombieri-Vinogradov theorem. Hypothesis \ref{taille de A} follows from the prime number theorem. As $\h(p)=\K=0$, Hypothesis \ref{hypothèse sur h} is trivial. Hypothesis \ref{bombieri friedlander iwaniec} follows from Theorem 9 of \cite{BFI}.

We now compute $\mu_{\K}(a,M)$. As $\h(p^e)=0$, we have $\omega_{\h}(a)=\omega(a)$, the number of prime factors of $a$. Thus, Remark \ref{remarque k=0} gives

$$\mu_0(a,M)=\begin{cases}
-\frac {1} 2  \log M &\text{ if } a= \pm 1,    \\
		-\frac 12 \left( 1-\frac 1p\right) \log p  &\text{ if } a= \pm p^e \\
		0 &\text{ if } \omega(a) \geq 2     
\end{cases}$$
so an application of Theorem \ref{theoreme principal} gives the result with a weaker error term. A better version of Proposition \ref{proposition evaluation de la somme arithmetique} follows from Huxley's subconvexity result \cite{huxley}, yielding the stated error term (see \cite{fiorilli} for a more precise proof).
 \end{proof}

\begin{proof}[Proof of Theorem \ref{resultat formes quad generales}]

Let $Q(x,y):=\alpha x^2+\beta xy + \gamma y^2$ be a binary quadratic form, where $\alpha,\beta$ and $\gamma$ are integers such that $\alpha>0$, $(\alpha,\beta,\gamma)=1$ and $d:=\beta^2-4\alpha\gamma < 0$ (so $Q(x,y)$ is positive definite). Note that the set of $d$ for which $d \equiv 1,5,9,12,13 \bmod 16$ includes a large subset of all fundamental discriminants. 
The set of bad primes is $\S :=  \{p: p \mid 2d \}$ in this case. Since $\S \neq \emptyset$, we will need to modify the proof of Theorem \ref{theoreme principal}. We define
$$ \chi_d := \left( \frac{4d}{\cdot} \right).$$ 
Note that for $(n,2d)=1$, we have the equalities
\begin{equation} r_d(n)=\sum_{m\mid n} \chi_d(m) = \prod_{\substack{p^k\parallel n: \\ \chi_d(p)=1}} (k+1)\prod_{\substack{p^k\parallel n: \\ \chi_d(p)=-1, \\ k \text{ odd}}}0.
\label{decomposition de r_d(n)}
\end{equation}
\label{sous-section r(n)}
An intuitive argument suggests that
\begin{equation*} \A(x;q,a) \sim \frac{R_a(q)}{q^2} \A(x), \end{equation*}
where
\begin{equation} R_a(q) := \#\{ 1\leq x,y \leq q : Q(x,y) \equiv a \bmod q \}. \label{définition de R_a}\end{equation}
As this is a classical result, we leave its proof, as well as several other classical facts about binary quadratic forms, to Appendix \ref{section généralités sur formes quadratiques}.
The function 
$$ \g_a(q):= \frac {R_a(q)}{q^2}$$
is actually multiplicative (see Lemma \ref{multiplicativité de R_a(q)}), and Lemma \ref{lemme evaluation de R_a forme quad generale} shows that for $p\nmid 2d$, $\g_a$ is given as in \eqref{definition de g avec h} with
$$ \h(p^e):= \begin{cases}
    	1+e\left(1-\frac 1 p\right) &\text{ if } \chi_d(p)=1 \\
		\frac 1p &\text{ if } \chi_d(p)=-1 \text{ and } 2 \nmid e \\
		1 & \text{ if } \chi_d(p)=-1 \text{ and } 2 \mid e,
            \end{cases}
 $$
and for $p\mid 2d$, $R_a(p^e)$ is given as in \eqref{determination de R_a(p^e)} and \eqref{determination de R_a(2^e)}. Since we are looking at large moduli, we need to use a result of Plaksin (Lemma 8 of \cite{plaksin2}), which asserts that 
\begin{equation}
\label{plaksin general}
\A(x;q,a) =  \g_a(q) \A(x) + E(x,q),
\end{equation}
where $E(x,q)\ll_{a,\epsilon} (x/q)^{\frac 34+\epsilon}$ if $q \leq x^{\frac 13}$, and $E(x,q)\ll_{a,\epsilon} x^{\frac 23+\epsilon}q^{-\frac 12}$ if $x^{\frac 13}<q \leq x^{\frac 23}$. Summing \eqref{plaksin general} over $q\leq x^{\frac 12}$, we get that the hypotheses \ref{bombieri-vinogradov} and \ref{bombieri friedlander iwaniec} hold with $\R(x):= x^{\frac 12}$ and $\LL(x):= x^{\lambda}$, provided $\lambda<\frac 1{12}$. (Note that in the case $\beta=0$, we can take the wider range $\lambda <\frac 18$, using Lemma 20 of \cite{plaksin}.)
Hypothesis \ref{taille de A} follows from Gauss' estimate:
$$ \A(x) = A_{Q} x + O(x^{\frac 12}), $$
where $A_{Q}$ is the area of the region $\{(x,y) \in \mathbb R_{\geq 0}^2 : Q(x,y) \leq 1\}$.
Let us turn to Hypothesis \ref{hypothèse sur h}. For $p\nmid 2d$, 
$$\h(p) = \begin{cases}
          2-\frac1p & \text{ if } \chi_d(p)=1 \\
          \frac1p & \text{ if } \chi_d(p)=-1, \\
         \end{cases}
 $$
so we set $\K:=1$ and
$$  \sum_{p\notin \S} \frac{\h(p)-\K}{ p} = \sum_{p\nmid 2d} \frac{\chi_{-d}(p)}{ p}+O(1) < \infty$$
by the prime number theorem for $\psi(x,\chi_{-d})$ (see \cite{davenport}). Moreover,
\begin{align} 
\begin{split}\sum_{p\notin \S} \frac{(\h(p)-\K) (\log p)^{n+1}}{p^{1+it}}  &= O(1) + (-1)^{n+1}\left(\frac{L'}{L} \right)^{(n)}(1+it, \chi_{-d})  \\
 & \ll_{d,n} (\log(|t|+2))^{n+2},
\end{split}  
\label{borne sur L(schi)}
\end{align}
this last bound following from Cauchy's formula for the derivatives combined with the classical bound for $\frac{L'(s,\chi)}{L(s,\chi)}$ in a zero-free region (see Chapter 19 of \cite{davenport}). As in the proof of Théorème II.3.22 of \cite{tenenbaum}, we can deduce from \eqref{borne sur L(schi)} that (setting $\eta:= 1/\log^2(|t|+2)$)
\begin{align*} 
\sum_{p\notin \S} \frac{\h(p)-\K }{p^{1+it}}  + O(1)&= \log L(1+it, \chi_{-d}) = \int_{1+it+\eta}^{1+it} \frac{L'(s, \chi_{-d})}{L(s, \chi_{-d})} ds+ \log L(1+it+\eta, \chi_{-d}) \\
 &\ll \eta \log^2(|t|+2) + \log \zeta(1+\eta) =2 \log\log(|t|+2)+O(1).
\end{align*}

Having proven hypotheses \ref{bombieri-vinogradov}, \ref{taille de A}, \ref{hypothèse sur h} and \ref{bombieri friedlander iwaniec}, we now proceed to prove an analogue of Theorem \ref{theoreme principal} (since the set $\S$ is non-empty). In the proof of Lemma \ref{transformée de Mellin}, we need to change the definition of $\mathfrak S_2(s)(=\mathfrak S_1(s))$ to (remember that $(a,2d)=1$)
\begin{multline*}\mathfrak S_2(s)= \left( \left(1-\frac 1{2^{s+1}}\right)\left(1+\frac{R_a(2)}{2^{s+2}}\right) + \frac{R_a(4)}{4}\frac 1{2^{2s+2}}\right)\prod_{\substack{p\mid d \\ p\neq 2}}\left( 1-\frac 1{p^{s+1}} + \frac{R_a(p)}{p^{s+2}}\right)  \\
\times \prod_{\substack{p^f \parallel a \\ f\geq 1 \\ p\notin \S }} \left[ \left( 1+\frac{\h(p)}{p^{s+1}}+...+\frac{\h(p^{f})}{p^{f(s+1)}} \right) \left(1-\frac 1{p^{s+1}} \right) + \frac{ \h(p^f)-\h(p^{f+1})/p}{1-1/p } \frac 1 {p^{(f+1)(s+1)}} \right],
\end{multline*}
(We also need to change the condition on the product defining $Z_5(s)$ to $p\nmid 2ad$) so
\begin{align*} \mathfrak S_2(-1) &= \frac{R_a(4)}{4}  \prod_{\substack{p\mid d \\ p\neq 2}} \frac{R_a(p)}{p}\prod_{\substack{p^f \parallel a \\ f\geq 1 \\ p\notin \S }} \frac{\h(p^f)-\h(p^{f+1})/p}{1-1/p} \\
&=  \frac{R_a(4)}{4}  \prod_{\substack{p^f \parallel d \\ p\neq 2}} \frac{R_a(p^f)}{p^f}\prod_{\substack{p^f\parallel a: \\ \chi_d(p)=1}}\left( 1-\frac 1p\right) (f+1)\prod_{\substack{p^f\parallel a: \\ \chi_d(p)=-1, \\ f \text{ even}}}\left( 1+\frac 1p\right) \prod_{\substack{p^f\parallel a: \\ \chi_d(p)=-1, \\ f \text{ odd}}}0 \\
&=\frac{R_a(4d)}{4d}  \prod_{p\mid a}\left( 1-\frac{\chi_d(p)}p\right)r_d(|a|),
\end{align*}
by \eqref{decomposition de r_d(n)} and Lemma \ref{lemme evaluation de R_a forme quad generale}.
We conclude that Theorem \ref{theoreme principal} holds with
$$ \mu_1(a,M)=-\frac{R_a(4d)}{4d}\cdot\frac{r_d(|a|)}{2L(1,\chi_d)}, $$
which gives the result (with a weaker error term) by Dirichlet's class number formula. To get the better error term $O_{\epsilon}\left( \frac 1{M^{1/3-\epsilon}}\right)$, one has to get a better estimate in Proposition \ref{proposition evaluation de la somme arithmetique}. To do this, we go back to the proof of Proposition \ref{k entier} and remark that (with the notation introduced there) 
$$ Z_5(s) = \prod_{p\nmid 2ad} \left(1-\frac{\chi_d(p)}{p^{s+2}} \right),$$
so
$$ Z(s)= \frac{\mathfrak S_3(s) \zeta(s+1) L(s+2,\chi_d)}{s(s+1)},$$
where 
$$ \mathfrak S_3(s) := \mathfrak S_2(s)\prod_{p\mid 2ad} \left(1-\frac{\chi_d(p)}{p^{s+2}} \right)^{-1}.$$
Since $Z(s)$ is a meromorphic function on the whole complex plane, we can shift the contour of integration to the left until the line $\Re (s)=-\frac 43+\epsilon$. A standard residue calculation combined with the convexity bound on $L(s,\chi_d)$ gives
$$\frac 1{2\pi i} \int_{(-1/2)}\frac{\mathfrak S_3(s) \zeta(s+1) L(s+2,\chi_d)}{s(s+1)} M^s ds = -\frac{\mu_1(a,M)}{M}+O_{\epsilon}\left( \frac 1{M^{4/3-\epsilon}}\right),$$ 
from which we conclude the result.
 \end{proof}

\begin{proof}[Proof of Theorem \ref{resultat dyadique poids 1}]
Set $\S:=\{2\}$, $\K:=\frac 12$ and $\LL(x):= (\log x)^{\lambda}$ with $\lambda < 1/5$. We first prove Hypothesis \ref{taille de A} using a refinement of a theorem of Landau. We have
\begin{equation} \A(x) = C \frac{x}{\sqrt{\log x}}\left(1+ O\left( \frac x{\log x} \right)\right),\label{Landau sans poids} 
\end{equation}
with
$$ C:=\frac 1{\sqrt 2} \prod_{p\equiv 3  \bmod 4} \left(1-\frac 1{p^2} \right)^{-\frac 12}.$$ 
(See for instance Exercice 240 of \cite{tenenbaum}). The distribution of $\A$ in the arithmetic progressions $a\bmod q$ with $(a,q)=1$ is uniform, however a result of the strength of Plaksin's \eqref{plaksin general} is far from being known. The best result so far for individual values of $q$ (in terms of uniformity in $q$) is due to Iwaniec \cite{iwaniechalf}, which proved using the semi-linear sieve that if $(a,q)=1$ and $a\equiv 1 \bmod (q,4)$, then
\begin{equation} \A(x;q,a) = \frac{(2,q)}{(4,q)q\gamma(q)} \A(x) \left( 1+O\left( \frac{\log q}{\log x}\right)^{1/5}\right),
\label{iwaniec theoreme} 
\end{equation}
where
$$ \gamma(q):=\prod_{\substack{p \mid q \\ p\equiv 3 \bmod 4}} \left(1+\frac 1p \right)^{-1}.$$
An easy computation using the arithmetic properties of $\a(n)$ shows that
$$ \A_{p^e}(x)= \begin{cases}
  \A\left( \frac x{p^{e+1}}\right) & \text{ if } p\equiv 3 \bmod 4 \text{ and } 2\nmid e \\
\A\left( \frac x{p^{e}}\right) & \text{ otherwise,}               
               \end{cases}
$$
and more generally,
\begin{equation} \A_d(x)=\A\left( \frac{\h(d)}{d} x \right),
\label{linearite de A_d}
\end{equation}
 with 
$$ \h(p^e) := \begin{cases}
  \frac 1p & \text{ if } p\equiv 3 \bmod 4 \text{ and } 2\nmid e \\
1 & \text{ otherwise.}               
               \end{cases}$$
This confirms that our choice of $\K=\frac 12$ was good, and Hypothesis \ref{hypothèse sur h} follows as in the proof of Theorem \ref{resultat formes quad generales}. Moreover, \eqref{iwaniec theoreme} can be extended to $(a,q)=d$ for any fixed odd integer $d>1$, by using the identity $ \A(x;q,a)=\A\left(\frac{\h(d)}{d} x;\frac qd,\frac ad\right)$, hence Hypothesis \ref{bombieri-vinogradov}* holds.
As we have shown every hypothesis, we turn to the calculation of the average $\mu_{\frac 12}(a,M)$ (which is never zero since $\K\notin \mathbb Z$). We need to modify the definition of $\mathfrak S_2(s)$, changing the local factor at $p=2$ to
$$ \left( 1-\frac 1{2^{s+2}}\right)^{1/2} \left( 1-\frac 1{2^{2s+2}} + \frac 1{2^{2s+3}}\right). $$
Doing so and proceeding as in the proof of Theorem \ref{theoreme principal}*, we get the result. 
\end{proof}

\begin{lemma}
\label{lemme k-tuple modifié}
 Suppose that $\mathcal H=\{a_1n+b_1,\dots a_kn+b_k\}$ is an admissible $k$-tuple of linear forms and $q,a$ are two integers such that $(q,a_ia+b_i)=1$ for $1\leq i\leq k$. Then the modified $k$-tuple $\tilde{\mathcal H}:=\{a_1(qm+a)+b_1,\dots a_k(qm+a)+b_k\}$ is also admissible. Moreover, 
$$ \mathfrak S(\tilde{\mathcal H}) = \prod_{p\mid q} \left( 1-\frac{\nu_{\mathcal H}(p)}{p}\right)^{-1} \mathfrak S(\mathcal H).$$
\end{lemma}
\begin{proof}
 First, since $\mathcal H$ is admissible, we have $(a_i,b_i)=1$ for $1\leq i\leq k$. Fix a prime $p$. We need to show that $\nu_{\tilde{\mathcal H}}(p)<p$. 
For a fixed $i$ we have either $p\mid a_i$, in which case $p\nmid b_i$ so $a_in+b_i\not \equiv 0 \bmod p$, or $p\nmid a_i$, in which case the only solution to $a_in+b_i\equiv 0 \bmod p$ is $n\equiv -a_i^{-1}b_i$. Hence, if $p\nmid a_i$, then there are only $\nu_{\mathcal H}(p)<p$ distinct possible values for $-a_i^{-1}b_i \bmod p$, thus regrouping these we can write
$$ \prod_{i=1}^k(a_in+b_i) \equiv C\prod_{i:p\mid a_i}b_i \prod_{j=1}^{\nu_{\mathcal H}(p)} (n+k_j)^{e_j} \bmod p,$$
where the $k_j$ are distinct integers, $e_j\geq 1$ and $p\nmid C$. Using this and the fact that $(a_i,b_i)=1$, we get 
$$ \prod_{i=1}^k(a_i(qm+a)+b_i) \equiv D \prod_{j=1}^{\nu_{\mathcal H}(p)} (qm+a+k_j)^{e_j} \bmod p,$$
with $p\nmid D$. If $p\nmid q$, then this has exactly $\nu_{\mathcal H}(p)<p$ solutions, therefore $\nu_{\tilde{\mathcal H}}(p)<p$. Otherwise, this becomes
$$ \prod_{i=1}^k(a_i(qm+a)+b_i) \equiv \prod_{i=1}^k(a_ia+b_i) \not \equiv 0 \bmod p$$
since $(q,a_ia+b_i)=1$ for $1\leq i\leq k$. We conclude that $\tilde{\mathcal H}$ is admissible. The calculation of $\mathfrak S(\tilde{\mathcal H})$ follows easily.
\end{proof}

\begin{proof}[Proof of Theorem \ref{resultat ktuplets}]
Define $\S:=\emptyset$ and
$$ \a(n):= \prod_{\Ll \in \mathcal H} \Lambda(\Ll(n)) = \Lambda(a_1n+b_1) \Lambda(a_2n+b_2) \cdots \Lambda(a_k n+b_k).$$
In our context, some assumptions of Section \ref{cadre} do not hold. The reason is that the asymptotic for $\A(x;q,a)$ depends on $(q,\P(a;\mathcal H))$ rather than depending only on $(q,a)$. 
The correct conjecture in this case is that for integers $a$ and $q$ such that $(q,\P(a;\mathcal H))=1$, (see \cite{kawada}\footnote{Kawada imposes the additional condition that $R(\mathbf b):=\prod_{j=1}^k |a_j| \prod_{1\leq i,j\leq k} |a_ib_j-a_jb_i|$ is non-zero. However, we assume that our linear forms are admissible, distinct and $a_i\geq 1$; one can show that this implies $R(\mathbf b)\neq 0$.})
$$ \A(x;q,a) \sim  \frac{\A(x)}{ q\gamma(q)} , $$
with
$$ \gamma(q):= \prod_{p\mid q} \left( 1- \frac {\nu_{\mathcal H} (p)}p \right).$$
This actually follows from the Hardy-Littlewood conjecture, by taking the modified $k$-tuple of linear forms $\tilde{\mathcal L}_i(m):=a_i(qm+a)+b_i=qa_im+aa_i+b_i$, which is admissible if $\mathcal H$ is and $(q,\P(a;\mathcal H))=1$ (see Lemma \ref{lemme k-tuple modifié}). Using this idea, we get that the assumption of \eqref{Hardy-Littlewood} holding uniformly for $|a_i|\leq \LL(x)^{1+\delta}$ implies Hypothesis \ref{bombieri-vinogradov}*. We now prove an analogue of Proposition \ref{proposition evaluation de la somme arithmetique}.
Defining $$ Z_{\mathcal H}(s):= \sum_{n} \frac{\f_a(n)}{ n^{s+1}\gamma(n)}, $$
where
$$\f_a(q) := \begin{cases}
              1 &\text{ if } (\mathcal \P(a;\mathcal H),q)=1  \\
		0 & \text{ otherwise, }
             \end{cases}
$$ one can compute that
\begin{align*}  Z_{\mathcal H}(s) =\mathfrak S_2(s) \zeta(s+1) \zeta(s+2)^k Z_{0}(s) 
\end{align*}
with $$ \mathfrak S_2(s) := \prod_{p \mid \mathcal \P(a;\mathcal H)} \left( 1-\frac 1{p^{s+1}} \right) \left( 1+\frac{\nu_{\mathcal H} (p)}{p-\nu_{\mathcal H} (p)} \frac 1{p^{s+1}}\right)^{-1}$$
and
$$ Z_{0}(s):= \prod_p \left( 1+\frac{\nu_{\mathcal H} (p)}{p-\nu_{\mathcal H} (p)} \frac 1{p^{s+1}}\right) \left( 1-\frac 1{p^{s+2}}\right)^k$$
which converges for $\Re s> -3/2$. Note that $Z_{\mathcal H}(s)$ has a simple pole at $s=0$. Also, $\mathfrak S_2(s)$ has a zero of order $\omega(\P(a,\mathcal H))$ at the point $s=-1$, and $Z_0(-1)= \mathfrak S (\mathcal H)^{-1}$, so $Z_{\mathcal H}(s)$ is of order $\omega(\P(a,\mathcal H))-k$ at this point. The function
$$\psi(s):=\frac {(2M)^s-M^s}{s+1}  +\left(\frac{x}{2M}\right)^s-\left(\frac{x}{M}\right)^s $$
vanishes to the second order at $s=0$. Combining all this information, we obtain by shifting the contour of integration to the left that
$$ \frac 1{2\pi i} \int_{(1)} Z_{\mathcal H}(s) \psi(s) \frac{ds}s =  \frac 1{2M} \left( \mu_{1-k}(a,M)(1+o(1)) + O\left( \frac 1{M^{\delta_k}}\right) \right),$$
where
$$ \mu_{1-k}(a,M):=         \begin{cases} \displaystyle        
 -\frac{1}{2\mathfrak S(\mathcal H)} \frac{(\log M)^{k-\omega(\P(a;\mathcal H))}}{(k-\omega(\P(a;\mathcal H)))!} \prod_{p\mid \P(a;\mathcal H)}  \frac{p-\nu_{\mathcal H}(p)}p \log p & \text{ if } \omega(\P(a;\mathcal H)) \leq k \\
0 & \text{ otherwise, }
                                
                               \end{cases}
         $$
and $\delta_k>0$ is a small real number (one can take $\delta_k=\frac 1{2+k}$). 
We conclude by proceeding as in the proof of Theorem \ref{theoreme principal}*. 

\end{proof}

In the case of twin primes (that is $\a(n):=\Lambda(n)\Lambda(n+2)$), we give an explicit description of all integers $a \geq -1$ (without loss of generality since $-a(-a+2) = a(a-2)$) for which $\mu_{-1}(a,M)\neq 0$ (note the occurrence of Mersenne and Fermat primes):

\vspace{.5cm}
\begin{center}
\begin{tabular}{|c|c|c|}
\hline
$a$ & $a(a+2)$ & $\omega(a(a+2))$ \\
\hline
-1 & -1 & 0 \\
1 & 3 & 1\\
2 & 8 & 1 \\
$p^e, p\neq 2 : p^e+2=q^f$ & $p^e q^f$ & 2 \\
$2^e : 2^{e-1}+1=q^f$ & $2^{e+1}(2^{e-1}+1) $ & 2\\
$2^{e}-2 : 2^{e-1}-1=q^f$ & $2^{e+1}(2^{e-1}-1)$ & 2  \\
\hline
\end{tabular}
\end{center}

\vspace{0.5cm}

\begin{proof}[Proof of Theorem \ref{resultat nombres rough}]
Define $\S:=\emptyset$ and $\LL(x):=(\log x)^{1-\delta}$. We split the proof in two cases, depending on the size of $y$.

\textbf{Case 1:} $\log y \leq (\log M)^{\frac 12-\delta}$. 
The fundamental lemma of combinatorial sieve (see \cite{de bruijn}) gives the following estimate, in the range $2\leq y \leq x^{o(1)}$:
\begin{equation} \A(x,y)=  x \prod_{p\leq y} \left( 1-\frac 1p\right) \left( 1+E(x,y)\right), 
\label{lemme fondamental du crible}
\end{equation}
where $E(x,y)\ll x^{-\frac 13}$ for $2\leq y < \frac{(\log x)^2}{16}$, and $E(x,y)\ll u^{-u} (\log y)^3$ for $\frac{(\log x)^2}{16} \leq y \leq x$, with the usual notation $u:=\frac{\log x}{\log y}$ (so $y^u=x$). This shows that Hypothesis \ref{taille de A} holds. One shows that
$$ \A_{d}(x,y):=\sum_{\substack{n\leq x \\ d\mid n}} \a_y(n) = \begin{cases}
                 \A\left( \frac x{d},y\right) &\text{ if } p\mid d \Rightarrow p\geq y \\
		 0 &\text{ else,}
                \end{cases}
$$
so we have $\A_{d}(x,y)= \A\left( \frac{\h_y(d)}{d} x,y\right)$, where
$$ \h_y(d):= \begin{cases}
                 1 &\text{ if } p\mid d \Rightarrow p\geq y \\
		 0 &\text{ else.}
                \end{cases}
$$
Wolke \cite{wolke} as shown a Bombieri-Vinogradov theorem for this sequence, which states that for any $A>0$, there exists $B=B(A)$ such that for any $Q\leq x^{\frac 12}/\log^B x$, we have, uniformly in the range $y\leq \sqrt x$,
\begin{equation} \sum_{q\leq Q} \max_{(a,q)=1} \max_{z\leq x} \left| \sum_{\substack{n\leq z \\ n \equiv a\bmod q}} \a_y(n)- \frac 1{\phi(q)} \sum_{\substack{n\leq z \\ (n,q)=1}} \a_y(n) \right|  \ll \frac x{\log^A x}.
\label{BV pour rough}
 \end{equation}
(Notice that if $(a,q)>1$, then $\A(x,y;q,a)$ is bounded.) We will only use this for $Q=2\LL(x)$, so from now on we suppose that $q\leq 2(\log x)^{1-\delta}$. Arguing as in Section \ref{cadre}, we have for $\frac x{2\LL(x)}\leq z\leq x$ that
\begin{align*}
 \frac 1{\phi(q)} \sum_{\substack{n\leq z \\ (n,q)=1}} \a_y(n) &=  \frac 1{\phi(q)} \sum_{d\mid q} \mu(d) \A_{d}(z,y) 
= \frac 1{\phi(q)} \sum_{d\mid q} \mu(d) \A\left( \frac{\h_y(d)}{d}z,y\right)  \\
&= \frac{\A(z,y)}{\phi(q)}  \sum_{d\mid q} \frac{\h_y(d)\mu(d) }{d}\left(1 + E_{d;q}(z,y)\right)
=  \frac{\A(z,y)}{q\gamma_y(q) }\left(1 + O(x^{-\frac 13+o(1)})\right),
\end{align*}
since by \eqref{lemme fondamental du crible}, in the range $d\leq q \leq (\log x)^{1-\delta}$ we have
\begin{equation*} E_{d;q}(z,y)\ll \left(\frac dz \right)^{\frac 13} \ll x^{-\frac 13+o(1)}.
\end{equation*}
Summing this over $q\leq 2\LL(x)$ and using \eqref{BV pour rough}, we get that 
\begin{equation} \sum_{q\leq 2\LL(x)} \max_{(a,q)=1} \max_{\frac x{2\LL(x)}\leq z\leq x} \left| \A(z,y;q,a)- \frac {\A(z,y)}{q\gamma_y(q)} \right|  \ll \frac {\A(x,y)}{\LL(x)^{1+\delta}}.
\label{BV a devenir! rough}
 \end{equation}
Having a Bombieri-Vinogradov theorem in hand, we now prove an analogue of Proposition \ref{proposition evaluation de la somme arithmetique}. A straightforward computation shows that
\begin{align} 
\label{nombres rough repr 1}
Z_{\A}(s):= \sum_{\substack{n\geq 1\\(n,a)=1}} \frac{1}{n^{s+1}\gamma_y(n)} &= \zeta(s+1)\prod_{p\mid a} \left(1-\frac 1{p^{s+1}} \right) \prod_{\substack{p\nmid a \\ p< y}} \left(1+\frac 1{(p-1)p^{s+1}} \right) \\
\label{nombres rough repr 2}
&=\mathfrak S_a(s)\zeta(s+1)\zeta(s+2)Z_{11}(s) \prod_{\substack{p\geq y}} \left(1+\frac 1{(p-1)p^{s+1}} \right)^{-1},
\end{align}
where 
$$ \mathfrak S_a(s):= \prod_{p\mid a} \left(1-\frac 1{p^{s+1}} \right) \left( 1+\frac{1}{(p-1)p^{s+1}}\right)^{-1},$$
$$ Z_{11}(s):=\prod_p \left( 1+\frac 1{(p-1)p^{s+2}}-\frac 1{(p-1)p^{2s+3}}\right).$$
We will now use representation \eqref{nombres rough repr 1}. Representation \eqref{nombres rough repr 2} will be useful for larger values of $y$, since then $\prod_{p<y}\left(1-\frac 1{p^{s+2}}\right)^{-1}$ behaves like $\zeta(s+2)$ on the line $\Re (s)=-1+\frac 1{\log M}$. Note that by \eqref{nombres rough repr 1}, $Z_{\A}(s)$ is defined on the whole complex plane, except at $s=0$. As before, we need to compute the integral
$$ I:=\frac 1{2\pi i} \int_{(2)} Z_{\A}(s)\psi(s) \frac{ds}s,$$
where $$\psi(s):=\frac {(2M)^s-M^s}{s+1}  +\left(\frac{x}{2M}\right)^s-\left(\frac{x}{M}\right)^s, $$
which has a double zero at $s=0$, so 
\begin{align*} I&= \frac 1{2\pi i} \int_{(-1/2)} Z_{\A}(s)\psi(s) \frac{ds}s \\
&= \frac 1{2\pi i} \int_{(-1/2)} Z_{\A}(s) ((2M)^s-M^s) \frac{ds}{s(s+1)}+O_{a,\epsilon}\left( \left(\frac{M}{x}\right)^{\frac 12-\epsilon} \log y \right),
\end{align*}
by the same arguments as in Lemma \ref{partie négligeable} (and Merten's theorem). We now proceed as in the proof of Proposition \ref{k entier}. Moving the contour of integration to $\Re (s)=\sigma=-1+\frac 1{\log M}$ and using the bounds $Z_{\A}(\sigma+it) \ll_{\epsilon} (|t|+1)^{\frac 12+\epsilon} \log y$ and $Z'_{\A}(\sigma+it)\ll_{\epsilon} (|t|+1)^{\frac 12+\epsilon}(\log y)^2$ for $|t|\geq 2$ (by Cauchy's theorem for the derivatives), we can deduce that
$$ I= \frac{2\mu_y(a,M)-\mu_y(a,2M)}{2M}\left(1+ O_{a}\left( \frac{(\log y)^2 \log \log M}{\log M}  \right)\right)+o(1), $$
where 
$$ \mu_y(a,M):= \begin{cases}
               -\frac{1}{2}\prod_{p< y}\left( 1-\frac 1p\right)^{-1} & \text{ if } a=\pm 1 \\
	0 & \text{ else.}
              \end{cases}
$$
We conclude the proof in the same lines as that of Theorem \ref{theoreme principal}*.
\textbf{Case 2:} $\LL^{(1+\delta)\log\log \LL}\leq  y \leq \sqrt x $. Note that it is sufficient to consider this range, since $\LL^{(1+\delta)\log\log \LL}<(\log x)^{\log\log\log x}$. We have
$$ \A(x,y) = \frac{x \omega(u)}{\log y}\left(1 + O\left( \frac 1{\log y}\right)\right),$$
where $u:=\frac{\log x}{\log y}$ and $\omega(u)$ is Buchstab's function (see Théorème III.6.4 of \cite{tenenbaum}). Therefore, we can use the properties of $\omega(u)$ to show that in the range $\frac 1{\LL(x)} \leq z \leq 1+\delta$,
$$ \frac{\A(zx,y)}{A(x,y)} = z\frac{\omega\left(u-O\left(\frac {\log \LL}{\log y} \right)\right)}{ \omega(u)} \left(1 + O\left( \frac 1{\log y}\right) \right) = z \left(1 + O\left( \frac { \log \LL}{\log y}\right) \right),$$
hence Hypothesis \ref{taille de A} holds.
Now, since $q\leq 2\LL(x) < y$, we have the equality
$$ \frac 1{\phi(q)} \sum_{\substack{n\leq x \\ (n,q)=1}}\a_y(n)=\frac 1{\phi(q)} \sum_{\substack{n\leq x}}\a_y(n)= \frac{\A(x,y)}{q\gamma_y(q)},$$ 
thus using \eqref{BV pour rough} we conclude that Hypothesis \ref{bombieri-vinogradov}* holds. We now turn to an analogue of Proposition \ref{proposition evaluation de la somme arithmetique}, which we prove using \eqref{nombres rough repr 2}. We need an estimate for
\begin{align*} I&= \frac 1{2\pi i} \int_{(-1/2)} Z_{\A}(s)\psi(s) \frac{ds}s \\
&= \frac 1{2\pi i} \int_{(-1/2)} Z_{\A}(s) ((2M)^s-M^s) \frac{ds}{s(s+1)}+O_{a,\epsilon}\left( \left(\frac{M}{x}\right)^{\frac 12-\epsilon} \right),
\end{align*}
since on the line $\sigma=-1+\frac 1{\log M}$, we have the bound
\begin{align*}\prod_{\substack{p\geq y}} \left(1+\frac 1{(p-1)p^{s+1}} \right)^{-1} &\ll  \prod_{\substack{p \geq \LL^{(1+\delta)\log\log \LL}}} \left(1+\frac {C_1}{p(\log p)^{1+\delta}} \right) \ll 1,
\end{align*}
and similarly for the derivative of this product. We now study the function $Z(s):=\frac{Z_{\A}(s)}{s(s+1)}$. Using the bounds we just proved, we get that for $s=-1+\frac 1{\log M}+it$ with $|t|\geq 2$,
$$ |Z(s)|,|Z'(s)| \ll_{\epsilon} (|t|+1)^{-\frac 32+\epsilon} .$$
 If $\omega(a)\geq 2$, then $Z(s)$ and $Z'(s)$ are bounded near $s=-1$ and we conclude that $I=o(1)$. If $\omega(a)= 1$, then we define 
$$ Z_{12}(s):= Z(s)-\frac{c(M,y)}{s+1},$$
where
$$ c(M,y):= -\frac 12\frac{\phi(a)}a \prod_{p\mid a} \log p\prod_{\substack{p\geq y}} \left(1+\frac 1{(p-1)p^{\frac 1{\log M}}} \right)^{-1}.$$
One sees that for $s$ close to $-1$ with $\Re (s)=\frac 1{\log M}$, 
$$ \prod_{\substack{p\geq y}} \left(1+\frac 1{(p-1)p^{s+1}} \right)^{-1} = (1+O(|s+1|))\prod_{\substack{p\geq y}} \left(1+\frac 1{(p-1)p^{\frac 1{\log M}}} \right)^{-1}, $$
hence $ |Z'_{12}(s)| \ll \frac 1{|s+1|},$ and thus 
\begin{equation} I=-\frac 12\frac{\phi(a)}{a}\prod_{p\mid a} \log p\prod_{\substack{p\geq y}} \left(1+\frac 1{(p-1)p^{\frac 1{\log M}}} \right)^{-1}\left(1+o(1)\right).
\label{evaluation de I rough}
\end{equation} 
If $a=\pm 1$, then we take $ Z_{13}(s):= Z(s)-\frac{c(M,y)}{(s+1)^2},$ and since $Z_{13}(s) \ll \frac 1{|s+1|}$, we get that \eqref{evaluation de I rough} holds.
Finally, in our range of $y$,
$$ \prod_{\substack{p\geq y}} \left(1+\frac 1{(p-1)p^{\frac 1{\log M}}} \right)^{-1}= 1+O\left( \frac 1{\log y}\right). $$
\end{proof}

\appendix
\section{Generalities on binary quadratic forms}
\label{section généralités sur formes quadratiques}
In this section we review several classical facts about the distribution of positive definite binary quadratic forms $Q(x,y)=\alpha x^2+\beta xy+\gamma y^2$ in arithmetic progressions. We recall the notations $d=\beta^2-4\alpha \gamma$, $\S = \{ p\mid 2d \}$, $\chi_d= \left( \frac{4d}{\cdot}\right)$ and 
\begin{equation*} R_a(q) = \#\{ 1\leq x,y \leq q : Q(x,y) \equiv a \bmod q \}.\end{equation*}

\begin{lemma}
The function $R_a(q)$ is multiplicative as a function of $q$.
\label{multiplicativité de R_a(q)}
\end{lemma}

\begin{proof}
 
Define $S_a(q) := \{ (x,y) \in (\mathbb{Z}\cap [1,q])^2 : Q(x,y) \equiv a \bmod q \}$ and let $q_1,q_2$ be two coprime integers. The "reduction mapping"

\begin{align*} S_a(q_1q_2) &\rightarrow S_a(q_1) \times S_a(q_2) \\
(x,y)\bmod q_1q_2 &\mapsto ((x,y) \bmod q_1,(x,y) \bmod q_2) 
\end{align*}
is a bijection by the Chinese remainder theorem.
\end{proof}

\begin{lemma}

Take $Q(x,y):=x^2-dy^2$ with $d \equiv -1 \bmod 4$, and let $a\neq 0$ be a fixed integer such that $(a,2d)=1$.
We have that $$\frac{R_a(q)}{q^2}= \frac{\f_a(q)}{q\gamma(q)},$$
where
$$ \gamma(q) := \prod_{p\mid q} \left(1-\frac{\chi_d(p)}p \right)^{-1}$$
and $\f_a(q)$ is a multiplicative function defined on primes as follows.

For $p\nmid 2ad$, $\f_a(p^e):=1$. For $p^f \parallel a$ with $f\geq 1$ (so $p\nmid 2d$), 

\begin{equation} \f_a(p^e) := \begin{cases}
                 e +1+\frac1{p-1} &\text{ if } \chi_d(p) = 1, e\leq f \\
f+1 & \text{ if } \chi_d(p)=1, e>f \\
	\frac 1{p+1} & \text{ if } \chi_d(p)=-1, e\leq f, 2 \nmid e \\
  1-\frac 1{p+1} & \text{ if } \chi_d(p)=-1, e\leq f, 2\mid e  \\
               0 &\text{ if } \chi_d(p)=-1, e>f, 2\nmid f \\
	1 & \text{ if } \chi_d(p) = -1,e>f,  2 \mid f. \\
        \end{cases}
\label{f_a forme quad}
\end{equation}

For $p\mid 2d$ (so $p\nmid a$), 
\begin{equation} \f_a(p^e) = \begin{cases}
	1+\left( \frac a{p}\right) & \text{ if } p\neq 2 \\
	1+\left( \frac{-4}{a}\right) & \text{ if } p=2, e \geq 2 \\
	1 & \text{ if } p=2, e=1.
        \end{cases}
\label{f_a somme de 2 carrés singulière}
\end{equation}
\label{lemme evaluation de R_a}
\end{lemma}

\begin{proof}

By Lemma \ref{multiplicativité de R_a(q)}, it is enough to show that for any prime $p$ and integer $e\geq 1$,

\begin{equation} \frac{R_a(p^e)}{p^{e}} = \frac{\f_a(p^e)}{ \gamma(p)}. 
 \label{chose a prouver avec R_a}
\end{equation}

\textbf{First case: $p\nmid 2d$.}

We will proceed as in section 2.3 of \cite{blomer}, by using Gauss sums. Writing $e(n):=e^{2\pi i n}$,
\begin{align*} R_a(p^e) &= \frac 1{p^e} \sum_{ 1\leq m \leq p^e} e\left(-m\frac{a}{p^e}\right) \left(\sum_{1\leq x \leq p^e} e\left(m \frac{x^2}{p^e}\right)  \right)\left(\sum_{1\leq y \leq p^e} e\left(-md \frac{y^2}{p^e}\right)  \right) \\
 &= p^e+\frac 1{p^e}\sum_{ 1\leq m \leq p^e-1} e\left(-m\frac{a}{p^e}\right) g(m;p^e)g(-md;p^e) 
\end{align*}
where $g(m;q):=\sum_{n=1}^q e(m n^2/q)$ is a Gauss sum. We have the following properties (see \cite{berndt}):
\begin{align}
\label{somme de gauss pour 1}
&\text{If } q \text{ is odd, then}  \hspace{2cm} g(1;q)^2=\left( \frac{-1}q\right)q. \\
\label{somme de gauss ramener a 1}
&\text{If } (q,m)=1, \text{ then}\hspace{1.55cm}  g(m;q)= \left( \frac{m}q\right)g(1;q).
\end{align}
As for Ramanujan sums, (see for example (3.3) of \cite{iwaniec})
\begin{equation}
 \label{somme de Ramanujan}
\sum_{\substack{m=1 \\ (m,q)=1}}^q e(ma/q) = \phi(q) \frac{\mu(q/(q,a))}{\phi(q/(q,a))}.
\end{equation}
Using these properties, we compute
\begin{align*} R_a(p^e) &=p^e+\frac 1{p^e} \sum_{g=1}^e \sum_{ \substack{1\leq m \leq p^e-1 \\ p^{e-g} \parallel m}  } e\left(-m\frac{a}{p^e}\right) g(m;p^e)g(-md;p^e)  \\
& = p^e+\frac 1{p^e}\sum_{g=1}^e \sum_{\substack{1\leq m' \leq p^g-1 \\ p\nmid m' }} e\left(-m'\frac{a}{p^g}\right) p^{2e-2g}g(m';p^g)g(-m'd;p^g) \\
&= p^e+p^e\sum_{g=1}^e \left(\frac{d}{p^g} \right) p^{-g}\sum_{\substack{1\leq m' \leq p^g-1 \\ p\nmid m' }}  e\left(-m'\frac{a}{p^g}\right) & \text{ by \eqref{somme de gauss pour 1} and \eqref{somme de gauss ramener a 1}} \\
&= p^e+p^e\sum_{g=1}^e \left(\frac{d}{p} \right)^g \left(1-\frac 1p\right) \frac{\mu(p^g/(p^g,a))}{\phi(p^g/(p^g,a))} & \text{by \eqref{somme de Ramanujan}},
\end{align*}
which shows (after a straightforward computation) that \eqref{chose a prouver avec R_a} holds for $p\nmid 2d$.

\textbf{ Second case: $p\mid 2d$, $p\neq 2$.}

In this case we have that $p\nmid a$, since $(a,\S)=1$. The number of solutions of $x^2-dy^2 \equiv a \bmod p$ is exactly $p\left(1+\left(\frac ap \right)\right)$. Moreover, such a solution must satisfy $x\not \equiv 0 \bmod p$, thus by Hensel's lemma we obtain that
$$\frac{R_a(p^e) }{p^e} =  1+\left(\frac ap \right).$$

\textbf{ Third case: $p=2$.}

In this case, $2\nmid a$. We have that $R_a(2)=2$. Reducing the equation $x^2-dy^2 \equiv a \bmod 2^e$ (using that $d\equiv -1 \bmod 4$), we get
$$ x\not \equiv y \bmod  2 , \hspace{2cm} x^2+y^2 \equiv a \bmod 4, $$
which shows that there are no solutions if $a\equiv 3 \bmod 4$. Suppose now that $a\equiv 1  \bmod 4$. For $e\geq 3$, an odd integer is a square $\bmod$ $2^e$ if and only if it is congruent to $1 \bmod 8$; in fact we have the following isomorphism:
$$\left( \mathbb Z / 2^e \mathbb Z\right)^{\times} \simeq \mathbb Z / 2\mathbb Z \times \mathbb Z / 2^{e-2} \mathbb Z.$$
 Using these well-known facts, we find the number of solutions to $x^2-dy^2\equiv a  \bmod 2^e$ such that $x$ is odd is
\begin{align*}
 &=4\# \{ y \bmod 2^e : d y^2+a \equiv 1  \bmod 8 \} \\
&=2^{e-1} \# \{ y \bmod 8 : y^2 \equiv d^{-1}(1-a)  \bmod 8 \} = 2^e
\end{align*}
since $d^{-1}(1-a)\equiv 0,4 \bmod 8$. Now the number of solutions of $x^2-dy^2\equiv a  \bmod 2^e$ such that $x$ is even is just the number of solutions of $y^2-d^{-1}x^2\equiv -d^{-1} a \bmod 2^e$ such that $y$ is odd, which as we have shown (and using that $-d^{-1}\equiv 1 \bmod 4$) is equal to $2^e$. We conclude that 
$$ \frac{R_a(2^e)}{2^e} = \begin{cases}
                           2 &\text{ if } a\equiv 1 \bmod 4 \\
			0 & \text{ if } a \equiv 3 \bmod 4.
\end{cases}
$$
\end{proof}

\begin{lemma}
Take $Q(x,y):=\alpha x^2 + \beta xy + \gamma y^2$ with $(\alpha,\beta,\gamma)=1$ and $d=\beta^2-4\alpha \gamma \equiv 1,5,9,12,13 \bmod 16$. Let $a\neq 0$ be a fixed integer with $(a,2d)=1$.
We have for $(q,2d)=1$ that $$\frac{R_a(q)}{q^2}= \frac{\f_a(q)}{q\gamma(q)},$$
where
$$ \gamma(q) := \prod_{p\mid q} \left(1-\frac{\chi_d(p)}p \right)^{-1}$$
and $\f_a(q)$ is defined as in Lemma \ref{lemme evaluation de R_a}. Moreover, for $p\mid 2d$, $p\neq 2$ (so $p\nmid a$), 

\begin{equation}
\frac{R_a(p^e)}{p^e}= \begin{cases}
			1+\left( \frac {\alpha a}{p}\right) & \text{ if } p\neq 2, p\nmid \alpha \\
			1+\left( \frac {\gamma a}{p}\right) & \text{ if } p\neq 2, p\nmid \gamma \\
\end{cases}
\label{determination de R_a(p^e)}
\end{equation}
and
\begin{equation}\frac{R_a(2^e)}{2^e}= \begin{cases}
			1 &\text{ if } 2\mid \beta,e=1 \\
                        1+\left( \frac{-4}{\alpha a}\right) & \text{ if } 2\mid \beta, 2\nmid \alpha, e \geq 2 \\
			1+\left( \frac{-4}{\gamma a}\right) & \text{ if } 2\mid \beta, 2\nmid \gamma, e \geq 2 \\
			\frac 12 & \text{ if } 2\nmid \beta, 2\mid \alpha \gamma \\
				\frac 32 & \text{ if } 2\nmid \alpha \beta \gamma. 
                        \end{cases}
\label{determination de R_a(2^e)}
\end{equation}

\label{lemme evaluation de R_a forme quad generale}
\end{lemma}

\begin{proof}
First write $Q(x,y)$ in four different ways:
\begin{align}
 \label{compl. de carre 1}
Q(x,y) &= \frac{1}{4\alpha}((2\alpha x+\beta y)^2-dy^2) \\
&=\frac 1{\alpha}\big(\big(\alpha x+\frac{\beta}2 y\big)^2-\frac d4 y^2\big)  \label{compl. de carre 1 divise par 2} \\
 \label{compl. de carre 2}
&= \frac{1}{4\gamma}((\beta x+2\gamma y)^2-dx^2) \\
 \label{compl. de carre 2 divise par 2}
&=\frac 1{\gamma}\big(\big(\gamma y+\frac{\beta}2 x\big)^2-\frac d4 x^2\big).
\end{align}

We will split in five distinct cases.

\textbf{Case 1: $p\nmid 2\alpha$. }
In this case, we use the representation \eqref{compl. de carre 1}. Note that the mapping $\phi_y: x\mapsto 2\alpha x+\beta y$ is an automorphism of $\mathbb Z/p^e\mathbb Z$, so 
$$ R_a(p^e) = \#\{ 1\leq x,y \leq p^e : x^2-dy^2 \equiv 4\alpha a \bmod p^e\}.$$
Going through the proof of Lemma \ref{lemme evaluation de R_a}, we see that 
$$ \frac{R_a(p^e)}{p^e} = \frac{\f_{4\alpha a}(p^e)}{\gamma(p)}=\frac{\f_{\alpha a}(p^e)}{\gamma(p)}\hspace{0.5cm}\left(=\frac{\f_{a}(p^e)}{\gamma(p)} \text{ if } p\nmid d \right).$$

\textbf{Case 2: $p\nmid 2\gamma$. }
In this case, we proceed in an analogous way to the first case, using the representation \eqref{compl. de carre 2} to get that
$$ \frac{R_a(p^e)}{p^e} = \frac{\f_{4\gamma a}(p^e)}{\gamma(p)}=\frac{\f_{\gamma a}(p^e)}{\gamma(p)}\hspace{0.5cm} \left(=\frac{\f_{a}(p^e)}{\gamma(p)} \text{ if } p\nmid d \right).$$

\textbf{Case 3: $p\mid \alpha$, $p\mid \gamma$, $p\neq 2$.}
In this case $p\nmid \beta$, so $p\nmid d$. Writing $X:=x+y$ and $Y:=y$, we compute that
$$ \alpha X^2 + \beta XY + \gamma Y^2 = \alpha x^2 + (2\alpha + \beta) xy + (\alpha+\beta+\gamma) y^2 =: \alpha' x^2 + \beta' xy + \gamma' y^2. $$
We have $p\mid \alpha'$, $p\nmid \beta '$ and $p\nmid \gamma '$, which reduces the problem to Case 2, and so
$$ \frac{R_a(p^e)}{p^e} = \frac{\f_{(\alpha+\beta+\gamma) a}(p^e)}{\gamma(p)} = \frac{\f_{a}(p^e)}{\gamma(p)}.$$

\textbf{Case 4.1: $p=2$, $2\mid \beta$.}
In this case, $d \equiv 0 \bmod 4$. We have that either $2\nmid \alpha$, or $2\nmid \gamma$. In the first event we use representation \eqref{compl. de carre 1 divise par 2}, which gives
$$ R_a(2^e)=\#\{ 1\leq x,y \leq 2^e : x^2-d'y^2 \equiv \alpha a \bmod 2^e \}$$
with $d':=\frac d4\equiv -1 \bmod 4$. Going back to the proof of Lemma \ref{lemme evaluation de R_a}, we get that 
$$ \frac{R_{a}(2^e)}{2^e} = \f_{\alpha a}(2^e). $$
In the event that $2\nmid \gamma$, the result is
$$ \frac{R_{a}(2^e)}{2^e} = \f_{\gamma a}(2^e). $$
Note that if $2\nmid \alpha \gamma$, then since $\frac d4 \equiv -1 \bmod 4$, we have $\alpha \equiv \gamma \bmod 4$, so 
$$\f_{\alpha a}(2^e) = \f_{\gamma a}(2^e) .$$

\textbf{Case 4.2: $p=2$, $2\nmid \beta$.}
In this case, $2\nmid d$ and $2\nmid a$. An easy application of Hensel's lemma in either of the variables $x$ or $y$ (since one of them has to be odd) yields 
$$ \frac{R_a(2^e)}{2^e} = \frac{R_a(2)}{2}, $$
and all the possibilities are contained in the following table.
\vspace{0.5cm}
\begin{center}
\begin{tabular}{|c|c|c|c|}
\hline
$\alpha \bmod 2$ & $\beta \bmod 2$ & $\gamma \bmod 2$ & $R_a(2)$ \\
\hline
0 & 1 & 0 & 1 \\
0 & 1 & 1 & 1 \\
1 & 1 & 0 & 1 \\
1 & 1 & 1 & 3 \\
\hline
\end{tabular}
\end{center}
\end{proof}

\end{document}